\documentclass[11pt,a4paper]{scrartcl}

\usepackage{a4wide}
\usepackage{latexsym}
\usepackage[USenglish]{babel}
\usepackage[utf8]{inputenc}
\usepackage[T1]{fontenc}
\usepackage{amssymb}
\usepackage{hyperref}
\usepackage[amsthm,thmmarks]{ntheorem}

\usepackage{mathrsfs}
\usepackage{graphicx}
\usepackage{amsmath}
\usepackage{mathdots}
\usepackage{amsfonts}
\usepackage{amsxtra}
\usepackage{url}
\usepackage{mathptmx}
\usepackage{helvet}
\usepackage{stmaryrd}
\usepackage{dsfont}
\usepackage{setspace}
\usepackage{enumerate}
\usepackage{MathDef}
\usepackage{verbatim}

\usepackage{color}
\usepackage{cleveref}
\usepackage{cancel}

\def\dif{\mathrm{d}}

\def\diag{\text{diag}}
\newcommand{\Sympp}{\mathbb{S}^{++}} 

\theoremstyle{plain}
\newtheorem{theorem}{Theorem}[section]
\newtheorem{lemma}[theorem]{Lemma}
\newtheorem{proposition}[theorem]{Proposition}
\newtheorem{definition}[theorem]{Definition}

\newtheorem{assumptionletter}{{\textbf{Assumption}}}

\theoremstyle{definition}

\newtheorem{remark}[theorem]{Remark}

\numberwithin{equation}{section}
\allowdisplaybreaks

\usepackage{titlesec}

\titleformat{\paragraph}{\normalfont\bfseries}{\theparagraph}{}{}
\titlespacing*{\paragraph}{0pt}{3.25ex plus 1ex minus .2ex}{0.1em}
\title{Cointegrated Continuous-time Linear \\[2mm] State Space and MCARMA Models}

\author{Vicky Fasen-Hartmann \setcounter{footnote}{1}\thanks{Institute of Stochastics, Englerstra{\ss}e 2,
D-76131 Karlsruhe, Germany.} \label{fnote}
\thanks{Financial
support by the Deutsche Forschungsgemeinschaft through the research
grant FA 809/2-2 is gratefully acknowledged.} \and  Markus Scholz \thanks{Allianz Lebensversichung-AG, Reinsburgstra{\ss}e 19, D-70197 Stuttgart, Germany. }}

\date{}

\begin{document}
%
\maketitle
%
\begin{abstract} \vspace*{-1.4cm}
In this paper we define and characterize cointegrated solutions of continuous-time linear state-space models.
A main result is that a cointegrated solution of a continuous-time linear state-space model can be represented as a sum
of a Lévy process and a stationary solution of a linear state-space model. Moreover, we prove that the class of cointegrated multivariate Lévy-driven autoregressive moving-average (MCARMA) processes, the continuous-time
analogues of the classical vector ARMA processes, is equivalent to the class of cointegrated solutions of continuous-time linear state
space models. Necessary conditions for MCARMA processes to be cointegrated are given as well extending
the results of Comte \cite{Comte1999} for MCAR processes.
The conditions depend only on  the autoregressive polynomial if we have a minimal model. Finally, we investigate  cointegrated continuous-time
linear state-space models observed on a  discrete time-grid and calculate their linear innovations.
Based on the representation of the linear innovations we derive an error correction form. The error correction form
uses an infinite linear filter in contrast
to the finite linear filter for VAR models.
\end{abstract}


\noindent
\begin{tabbing}
\\ \> Secondary:   62M10, 91G70
\end{tabbing}

\vspace{0.2cm}\noindent\emph{Keywords:} Canonical form, cointegration, error correction form, Granger representation theorem, Kalman filter, MCARMA process,  state-space model.

\section{Introduction}
%
%

Many time series do not behave in a stationary way, e.g., financial time series. Hence, non-stationary models are of particular interest in order to model such behavior. One  class of non-stationary processes are cointegrated processes. A $d$-dimensional
stochastic process $(Y(t))_{t\geq t_0}$  without deterministic component  is called  \emph{integrated}
if $(Y(t))_{t\geq t_0}$  is non-stationary but has stationary increments.
If additionally there exists a vector $\beta\not=0$, $\beta\in\R^d$, such that $(\beta^\mathsf{T} Y(t))_{t\geq t_0}$ is stationary
then $(Y(t))_{t\geq t_0}$ is called
\emph{cointegrated} with cointegration vector $\beta$.
The number of linear independent cointegration relations is called \emph{cointegration rank}
and the space spanned by all linear independent cointegration vectors is called \emph{cointegrating space}.
Examples of cointegrated time series are, e.g.,  exchange rates, foreign
currency spot and futures/forwards rates, stock prices within an
industry and interest rates in different countries (cf. Brenner and
Kroner~\cite{BrennerKroner95} and references therein).
 It was Clive Granger who showed that statistical inference of such non-stationary time series with the classical stationary methodology can lead to inadequate results.
 The seminal works by Granger~\cite{GRANGER1981} in 1981  and Engle and Granger~\cite{EngleGranger1987} in 1987  lay the foundation for the field of cointegration analysis in discrete time.
 One of the most important results is the Granger representation theorem, which connects the moving average, autoregressive and error correction representations of cointegrated
 time series. Johansen \cite{Johansen1991,Johansen2009} characterizes cointegration for VAR (vector autoregressive) processes precisely by making assumptions on the autoregressive polynomial in the Johansen-Granger representation theorem.

In this paper we study  cointegrated solutions of continuous-time linear stochastic state-space models and cointegrated multivariate  continuous-time ARMA$(p,q)$ (MCARMA$(p,q)$) processes, respectively which are the natural extension of VARMA (vector autoregressive moving average) processes from discrete to continuous time. It is well known that the class of stationary solutions of linear state-space models and the class of stationary MCARMA processes
are equivalent; see Schlemm and Stelzer~\cite[Corollary 3.4]{SchlemmStelzer2012}.
Both model classes are driven by a Lévy process.
An $\R^m$-valued stochastic process $L=(L(t))_{t\geq 0}$ is a L\'evy process  if $L(0)=0_m$ $\mathbb{P}$-a.s., it has stationary and independent increments and c\`adl\`ag sample paths; see the excellent monograph of Sato~\cite{sato1999levy} for details on Lévy processes.
Then
an $\R^d$-valued \emph{continuous-time linear state-space model} $(A,B,C,L)$ of dimension $N\in\N$ is characterized by an $\R^m$-valued Lévy process, a
transition matrix $A\in \R^{N\times N}$, an input matrix $B\in \R^{N\times m}$ and an observation matrix  $C\in \R^{d\times N}$.
It consists of the state equation
\begin{align}
\notag
\dif X(t) &=  A X(t) \dif t + B \dif L(t)
\intertext{and the observation equation}
\label{defObservEqCoint}
Y(t) &= C X(t) \quad  \text{ for } t \geq t_0\geq0.
 \end{align}
The state vector process $X=(X(t))_{t \geq t_0}$ is an $\R^N$-valued process and the output process $Y=(Y(t))_{t \geq t_0}$ is $\R^d$-valued.
Every solution of \eqref{defObservEqCoint} has the representation
\begin{align*}
Y(t)=C\exp(A(t-t_0))X(t_0) +C\int_{t_0}^t\exp(A(t-u))B\, \dif L(u).
\end{align*}
On the other hand, the idea behind the definition of an $\R^d$-valued MCARMA$(p,q)$ process  ($p>q$ positive integers)
 is that it is the solution to the differential equation
 \begin{eqnarray*}
     {P}(D)Y(t)={Q}(D)D L(t) \quad \mbox{ for } t\geq t_0,
\end{eqnarray*}
where $D$ is the differential operator with respect to $t$,
$    {P}(z):= I_{d\times d}z^p+P_1z^{p-1}+\ldots+P_{p-1}z+P_p$
is the autoregressive polynomial with $P_1,\ldots, P_p\in \R^{d\times d}$ and
    ${Q}(z):=Q_0z^q+ Q_1z^{q-1}+\ldots+ Q_{q-1}z+Q_q$
 is the moving average polynomial with $Q_0,\ldots,Q_q\in \R^{d\times m}$. However, since a Lévy process is not differentiable
this is not the formal
definition of an MCARMA process. 

The aim of this paper  is to characterize cointegrated solutions of linear state-space models in continuous-time   and cointegrated MCARMA processes and to relate both model classes.
Cointegration in continuous time started being of interest in the early 1990s with Phillips \cite{Phillips1991}. In this work, Phillips investigated stochastic differential equations driven by a differentiable stationary process. The connection between cointegrated discrete-time models and continuous-time models were  analyzed by Chambers  \cite{Chambers1999619}. The literature on Gaussian MCAR$(1,0)$ processes is rich, e.g.,
 \cite{KesslerRahbek2001,litrahbek2004,litstockmarrjacobsen1994}. One of the first departures of the Gaussian assumption and the order (1,0) was Comte \cite{Comte1999}; he derived a characterization of  continuous-time integrated and cointegrated processes, and in particular, he presented an error correction form and a characterization of cointegration for MCARMA$(p,0)$ processes.   The processes considered in Fasen \cite{Fasen2013,Fasen2014} are special cases of cointegrated MCARMA processes.
In the discrete-time case they are a few papers dealing with cointegrated discrete-time state space models as
Aoki~\cite{Aoki90}, Aoki and Havenner~\cite{AokiHavenner91}, a series of papers by Bauer and Wagner~\cite{BauerWagner2002b,BauerWagner2012} and the preprint series of Ribaritis and Hanzon~\cite{RibaritsHanzon2011,RibaritsHanzon20112}; see also the overview article of Wagner~\cite{Wagner2010}.
All of the papers assume that the state space model is already in innovation form which is not the case in our setting.

The paper is structured on the following way. We start with an introduction into basic properties of linear state space models in \Cref{Section 2.1}. In \Cref{sec:CointMCARMA} we derive cointegrated solutions of linear state-space models in continuous-time
and calculate some probabilistic properties of them.
An important result is that a cointegrated solution of a linear state-space model has a representation
as a sum of a non-stationary and a stationary process;
the non-stationary process is a Lévy process and the stationary process is a stationary solution of a different linear state-space model.
This characterization we use as definition for cointegrated  state-space processes. Moreover, we present canonical forms for cointegrated state
space processes. Canonical forms are fundamental for statistical inference
 to get identifiable models.

Next, in \Cref{Section:3}, we investigate cointegrated MCARMA processes. We prove that the class
of cointegrated MCARMA processes and the class of cointegrated linear state-space processes are equivalent complementing
 \cite[Corollary 3.4]{SchlemmStelzer2012} to the non-stationary case.
Furthermore, we derive an alternative characterization of cointegration for MCARMA processes extending
the results of Comte \cite{Comte1999} for MCAR models to general MCARMA models. A conclusion is that
the property of cointegration of an MCARMA process is determined by
 the  matrices $P_p$ and $P_{p-1}$ of the autoregressive polynomial if the model is minimal.

The last section,  \Cref{sec:KalmanFilterforCointegratedProcesses},  studies  cointegrated solutions of
linear state-space models sampled at a discrete-time grid. They are of particular interest in high-frequency data where the data
are modelled by a continuous-time process observed discretely.
First, we derive for the sampled process $Y^{(h)}=(Y(nh))_{n\geq n_0}$, $h>0$ fixed, a state space representation in discrete time. For statistical inference
a drawback of this representation is that the state process is not observable. Hence, we derive an innovation representation
of the sampled process by an application of the Kalman filter. 
Since we are in a non-stationary setting the choice of the starting value of the Kalman filter is not obvious.
Having this starting value we receive a very nice representation of the linear innovations and then also an error correction form.
  The derived error correction form is similar to the original error correction form presented by Engle and Granger \cite{EngleGranger1987} for VAR models. The main difference is that we have an infinite order linear filter instead of a finite linear filter.
The meaning of an infinite linear filter is not as obvious for the non-stationary process $Y$ since $Y(t)$  is only defined for $t\geq t_0$
and hence, $Y(nh)$ is only defined for $n\geq n_0= \lceil t_0/n\rceil$. Thus, we have to extend the definition of $Y(nh)$ to $n\leq n_0$ so that
the infinite linear filter makes sense.
 We show that the cointegration information is contained in parts of the filter and is thus not lost by  sampling and filtering. 
The error correction form enables us to analyze the short-run   as well as the long-run behaviour of $Y^{(h)}$.
It provides the basis for  quasi-maximum-likelihood estimation of cointegrated state space process as presented in  Fasen and Scholz \cite{Fasen:Scholz:2016b}.
An outlook is given in \Cref{Outlook}.

\subsubsection*{Notation} \vspace*{-0.3cm}

We use as norms the Euclidean norm $ \lVert\cdot \rVert$ in $\R^d$
and the Frobenius norm $ \lVert\cdot \rVert$ for
matrices, which is submultiplicative. $\Re(z)$ denotes the real part of a complex number $z\in\C$.
The matrix $0_{d\times s}$
is the zero matrix in $\R^{d\times s}$ and $I_{d}$ is the
identity matrix in $\R^{d\times d}$.
 For a matrix $A\in\R^{d\times d}$ we denote by $A^\mathsf{T}$ its transpose,
 by $\det(A)$ its determinant, by $\sigma(A)$ the set of all eigenvalues of $A$, by $\rank A$ its rank and by $\lambda_{\text{max}}(A)$ its largest eigenvalue.
For a matrix  $A \in \R^{d \times s}$ with $\rank~A=s$, $A^{\perp}$ is a $d\times (d-s)$ matrix with rank $(d-s)$ satisfying $A^\mathsf{T}A^{\perp}=0_{s\times(d-s)}$ and $A^{\perp \mathsf{T}}A=0_{(d-s)\times s}$. For two matrices $A,B$ we write $\diag(A,B)$ for a block diagonal matrix whose first block is the matrix $A$ and the second
block is the matrix $B$.
The space of all $m\times n$ real-valued matrices is denoted with $M_{m,n}(\R)$,  the set of $m$-dimensional symmetric positive-definite matrices is denoted by $\Sympp_m$
and $GL_N(\R):=\{A\in\R^{N\times N}:\det~A\not=0\}$ for some $N\in\N$.
Throughout this paper we will always assume
\begin{assumptionletter}
\label{ALevy}
  The Lévy process $L$ satisfies $\E L(1)=0_m$ and $\E\|L(1)\|^2<\infty$ and the covariance matrix  $\Sigma_L=\E[L(1)L(1)^\mathsf{T}]$ is non-singular.
\end{assumptionletter}
We extend the definition of a Lévy process on $\R$ by taking an independent copy $\wt L$ of $L$ and defining $L(t)=\wt L({-t-})$ for $t<0$.

\section{Preliminaries} \label{Section 2.1}

First, we present definitions related to linear state-space models which we  need subsequently. These definitions enable us to imply restrictions on the state-space model in order to achieve uniqueness in the output process and to define a cointegrated model. The definitions hold for both discrete-time as well as continuous-time state space models; see
\cite{Caines,hannandeistler2012}  for discrete-time models and \cite{Bernstein2009,Kailath,Sontag} for continuous-time models.

\begin{definition}
\label{defMinimalSSMcont}
Let $A\in M_{N,N}(\R)$, $B\in M_{N,m}(\R)$ and $C\in M_{d,N}(\R)$.
The matrix triple $(A,B,C)$ is called an \textbf{algebraic realization} of a rational matrix function $k\in M_{d,m}(\R\{z\})$ of dimension $N$ if $k(z)=C(zI_N-A)^{-1}B$.
The function $k:~z\mapsto C(zI_N-A)^{-1}B$ is called \textbf{transfer function} of the triple $(A,B,C)$. The triple
  $(A,B,C)$ is called \textbf{minimal} if there exists no other algebraic realization  with dimension smaller than $N$.
The dimension of a minimal realization of $k$ is called \textbf{ McMillan degree} of $k$.
\end{definition}
Thus, non-minimality is a source of non-uniqueness of the state-space model. Minimality guarantees that we consider only components of the state vector which are relevant for the output process.  If we have a non-minimal system there might be non-stationary components having no effect on the output process. Therefore, this property implies a one-to-one correspondence of the non-stationarity of the state process and the output process.
In the following we want to characterize minimal state-space models, therefore we need some more definitions.

\begin{definition} $\mbox{}$
\begin{itemize}
\item[(a)] A  state-space model $(A,B,C)$ of dimension $N$  is \textbf{observable} if the observability matrix
\begin{eqnarray*}
\mathcal{O}_{CA}:=\begin{pmatrix} C^\mathsf{T}&(CA)^\mathsf{T}&\ldots&(CA^{N-1})^\mathsf{T}\end{pmatrix}^\mathsf{T}\in M_{dN,N}(\R)
\end{eqnarray*}
has full rank, i.e. if $\rank(\mathcal{O}_{CA})=N$.
\item[(b)]
A state-space model $(A,B,C)$ of dimension $N$  is \textbf{controllable} if the controllability matrix
\begin{eqnarray*}
\mathcal{C}_{AB}:=\begin{pmatrix} B& AB & \ldots &A^{N-1}B \end{pmatrix}\in M_{N,mN}(\R)
\end{eqnarray*}
has full rank, i.e. if $\rank(\mathcal{C}_{AB})=N$.
\end{itemize}
\end{definition}
There exists several equivalent characterizations of observability and controllability (see Bernstein~\cite[Theorem 12.3.18 and Theorem 12.6.18]{Bernstein2009})
we repeat only one of them which we use in this paper.

\begin{lemma}  \label{Lemma:control}
For $\lambda\in\C$ define the matrices
\beao
    \mathcal{O}_{CA}(\lambda)=\left(\begin{array}{c}\lambda I_N-A \\ C\end{array}\right) \quad \mbox{ and } \quad \mathcal{C}_{AB}(\lambda)=\left(\begin{array}{c}\lambda I_N-A \quad \quad B\end{array}\right).
\eeao
\begin{itemize}
    \item[(a)] The state space model $(A,B,C)$ is observable iff $\rank  (\mathcal{O}_{CA}(\lambda))=N$ for any $\lambda\in\C$.
    \item[(b)]  The state space model $(A,B,C)$ is controllable iff $\rank (\mathcal{C}_{AB}(\lambda))=N$ for any $\lambda\in\C$.
\end{itemize}
\end{lemma}
\begin{proof}
(a) follows from~\cite[Proposition 12.3.13 and Theorem 12.3.18 ]{Bernstein2009}. Similarly
(b) is a conclusion of~\cite[Proposition 12.6.13 and Theorem 12.6.18]{Bernstein2009}.
\end{proof}

\begin{lemma}
\label{LemmaMinimality}
 Let  $(A,B,C,L)$ be a $d$-dimensional  linear state-space model with block representation
  \begin{align*}
A:=\begin{pmatrix}0_{c\times c} & 0_{c\times (N-c)}\\ 0_{ (N-c)\times c} & A_2\end{pmatrix}\in M_{N,N}(\R),\,
B:=\begin{pmatrix}B_1\\  B_2\end{pmatrix}\in M_{N, m}(\R) \quad \text{and} \quad
C:=\begin{pmatrix}C_1, & C_2 \end{pmatrix}\in M_{d, N}(\R)
\end{align*}
where $0\leq c\leq d$, $A_2\in M_{N-c,N-c}(\R)$, $C_1\in M_{d,c}(\R)$, $C_2\in M_{d,N-c}(\R)$, $B_1\in M_{c,m}(\R)$ and $B_2\in M_{N-c,m}(\R)$.
Furthermore, assume that $A_2$ is non-singular.
 Then $(A,B,C)$ is minimal if and only if $\rank\, B_1=\rank\, C_1=c$ and $(A_2,B_2,C_2)$  is minimal.
 \end{lemma}
\begin{proof}
 From  Hannan and Deistler \cite[Theorem 2.3.3]{hannandeistler2012} we already know that  $(A,B,C)$ is minimal iff it is controllable and observable.
 Therefore we use the characterisation of observability and controllability given in \Cref{Lemma:control} and distinguish the two cases $\lambda=0$ and $\lambda\not=0$.\\
\textbf{Case 1:} $\lambda=0$. Then

\begin{align*}
\rank( \mathcal{O}_{CA}(0))&=\rank\begin{pmatrix}  0_c & 0_{c\times (N-c)} \\ 0_{(N-c)\times c}  &-A_2 \\ C_1 & C_2 \end{pmatrix}=\rank\begin{pmatrix}  -A_2 &0_{(N-c)\times c}  \\   C_2 &C_1\end{pmatrix}\\
&=(N-c)+\rank(C_1+C_2A_2^{-1}0_{(N-c)\times c})=(N-c)+\rank C_1.
\end{align*}
The third equality follows by Bernstein \cite[Proposition 2.8.3]{Bernstein2009} and $A_2$ non-singular. Hence, \linebreak  $\rank ( \mathcal{O}_{CA}(0))=N$ iff $\rank C_1=c$.
Similarly we receive $\rank  (\mathcal{C}_{AB}(0))=N$ iff $\rank B_1=c$.\\
\textbf{Case 2:} $\lambda\not=0$. Then
\begin{align*}
\rank ( \mathcal{O}_{CA}(\lambda))&=\rank\begin{pmatrix} \lambda I_c & 0_{c\times (N-c)} \\ 0_{(N-c)\times c}  &\lambda I_{N-c}-A_2 \\ C_1 & C_2 \end{pmatrix}\\
&=c+\rank\left(\begin{pmatrix}\lambda I_{N-c} -A_2 \\ C_2\end{pmatrix}-\begin{pmatrix} 0_{(N-c)\times c}\\C_1 \end{pmatrix} \cdot \lambda^{-1} I_c \cdot 0_{c\times (N-c)}\right)
\\&= c+\rank\begin{pmatrix}\lambda I_{N-c} -A_2 \\ C_2\end{pmatrix}=c+\rank  (\mathcal{O}_{C_2A_2}(\lambda)).
\end{align*}
Again, we used Bernstein \cite[Proposition 2.8.3]{Bernstein2009}. Thus,  $\rank (\mathcal{O}_{CA}(\lambda))=N$ iff $\rank ( \mathcal{O}_{C_2A_2}(\lambda))=N-c$ for $\lambda\not=0$.
Moreover, since $A_2$ is non-singular and has full rank, $\rank ( \mathcal{O}_{C_2A_2}(0))=N-c$ as well. Due to \Cref{Lemma:control}, we finally get that
$\rank(\mathcal{O}_{CA}(\lambda))=N$ iff $(A_2,B_2,C_2)$ is observable.
In combination with Case 1 and \Cref{Lemma:control} that means  that $(A,B,C)$ is observable iff $\rank C_1=c$ and $(A_2,B_2,C_2)$  observable.
Similar calculations yield $\rank ( \mathcal{C}_{AB}(\lambda))=c+\rank ( C_{A_2B_2}(\lambda))$ for $\lambda\not=0$ and hence,
$(A,B,C)$  is controllable iff $\rank B_1=c$ and $(A_2,B_2,C_2)$  controllable.
\end{proof}
Last but not least, we give the formal definition of observational equivalence of state-space models.

\begin{definition}
\label{defObsevEquiv}
 A minimal  state-space model $(A,B,C)$ is called \textbf{observationally equivalent} to the minimal model $(\widetilde{A},\widetilde{B},\widetilde{C})$ if they give rise to the same transfer function.
\end{definition}
A conclusion of Hannan and Deistler \cite[Theorem 2.3.4]{hannandeistler2012} is that
  $(A,B,C)$ and $(\widetilde{A},\widetilde{B},\widetilde{C})$ are observationally equivalent iff there exists a nonsingular transformation matrix $T\in GL_N(\R)$ such that $A=T\widetilde{A}T^{-1}$, $B=T\widetilde{B}$ and $C=\widetilde{C}T^{-1}$. Such a transformation leads to a corresponding basis change of the state vector to $\widetilde{X}(t)=TX(t)$. Hence,
  \begin{align*} \label{eq3}
    Y(t)&=C\exp(A(t-t_0))X(t_0) +C\int_{t_0}^t\exp(A(t-u))B\, \dif L(u) \nonumber\\
        &=\wt C\exp(\wt A(t-t_0))\wt X(t_0) +\wt C\int_{t_0}^t\exp(\wt A(t-u))\wt B\, \dif L(u)
\end{align*}
is a solution of  the state space model  $(A,B,C)$ as well as $(\wt A,\wt B,\wt C)$ .

\section{Cointegrated state-space models}\label{sec:CointMCARMA} \label{Section:2}

 \subsection{Characterization of cointegrated state space models}


The aim of this section is to characterize  cointegrated solutions of linear state space models. For this purpose we introduce a convenient canonical form for a state-space model. A canonical form is an unique representative of the class of observationally equivalent linear state-space models. There is no similar transformation matrix other than the identity matrix or permutation matrix which preserves this structure. An example for a canonical form for stationary solutions of continuous-time state space models is the echelon canonical form (see  Schlemm and Stelzer \cite[Section 4.1]{SchlemmStelzer2012}).
 Our canonical form for the cointegrated continuous-time linear state-space model has a similar structure as the canonical form for the discrete-time linear state-space models  presented in Bauer and Wagner \cite[Theorem 2]{BauerWagner2012}.
The advantage of this canonical form is that the non-stationary and the stationary part are decoupled and can be transformed separately. Moreover, this form enables us to use existing results for stationary output processes of state-space models and L\'evy processes in the following. Before we can state the result we repeat the definition of a positive lower triangular  matrix:  a matrix $M=[m_{ji}]_{1\leq i\leq d,1\leq j\leq c}\in M_{d, c}(\R)$ with $\rank M=c$ is positive lower triangular if
there exists $j_1,\ldots,j_c\in\{1,\ldots,d\}$ with $1\leq j_1<\ldots<j_c\leq d$ and $m_{j_ll}>0$ and $m_{jl}=0$ for $j<j_l$.

\begin{theorem}
\label{ThmCanonFormCoint}
Let $(A,B,C,L)$ be a  $d$-dimensional minimal state-space model  which satisfies $\sigma(A)\subset\{(-\infty,0)+i\R\} \cup \{0\}$, and the algebraic and the geometric multiplicity of the eigenvalue zero is equal to $c$, $0\leq c\leq d$.
Then there exists a unique observationally equivalent minimal state-space representation given by
\begin{equation}
\begin{split}
\begin{pmatrix}\dif X_1(t)\\\dif X_2(t)\end{pmatrix}&=\begin{pmatrix}0_{c\times c} & 0_{c\times (N-c)}\\ 0_{ (N-c)\times c} & A_2\end{pmatrix}\begin{pmatrix}X_1(t)\\X_2(t)\end{pmatrix}+\begin{pmatrix}B_1\dif L(t)\\  B_2\dif L(t)\end{pmatrix},
\\[2mm]
\label{eqCanonicalFormSSMcont}
Y(t)&=C_1 X_1(t)+ C_2 X_2(t), \quad t\geq t_0,
\end{split}
\end{equation}
where
\begin{enumerate}
\item \label{item1}
$\sigma(A_2)\subset\{(-\infty,0)+i\R\}$,
\item \label{item2}
the matrix $C_1\in M_{d, c}(\R)$ satisfies $C_1^\mathsf{T}C_1=I_c$ and $C_1$ is a positive lower triangular matrix,
\item \label{item3}
the matrix $B_1\in M_{c, d}(\R)$ has rank $c$,
\item the stationary state-space model $(A_2,B_2,C_2)$ is  given in canonical form.
    \label{item4}
\end{enumerate}
Moreover, $Y$ has the representation
\begin{align*}
Y(t)=C_1X_1(t_0) +C_2\exp(A_2(t-t_0))X_2(t_0)
+C_1B_1[L(t)-L(t_0)] +\int_{t_0}^tC_2\exp(A_2(t-u))B_2\, \dif L(u).
\end{align*}
\end{theorem}
\begin{proof}
We define the following matrices \begin{align}
\label{eqABCprime}
A^*:=\begin{pmatrix}0_{c\times c} & 0_{c\times (N-c)}\\ 0_{ (N-c)\times c} & A_2\end{pmatrix},\quad
B^*:=\begin{pmatrix}B_1\\  B_2\end{pmatrix} \quad \text{ and } \quad
C^*:=\begin{pmatrix}C_1, & C_2 \end{pmatrix}.
\end{align}
For the existence of the representation \eqref{eqCanonicalFormSSMcont} we need to show that there exists a $T\in GL_N(\R)$ which transforms the state-space model $(A,B,C)$ to the desirable form $(A^*,B^*,C^*)$ satisfying all restrictions (i)-(iv). Afterwards, we have to show that this transformation matrix is unique which
results in the uniqueness of this representation.

\textbf{Existence:}
Due to the eigenvalue assumption on the matrix $A$ the upper part of $A^*$ is just the Jordan normal form corresponding to the eigenvalue zero.
 Thus there exists a transformation matrix $T^\prime\in GL_N(\R)$ such that
 $$T'A{T'}^{-1}=\begin{pmatrix}0_{c\times c} & 0_{c\times (N-c)}\\ 0_{ (N-c)\times c} & A_2'\end{pmatrix}=:A',$$
 where the eigenvalues of $A_2'$ coincide with the non-zero eigenvalues of $A$ which have by assumption strictly negative real parts.
 Otherwise $A_2'$
 is not specified yet. Further,
 \begin{align*}
 B':=T'B=:\begin{pmatrix} {B_1'}^\mathsf{T},  {B_2'}^\mathsf{T}\end{pmatrix}^\mathsf{T} \quad
 \mbox{and} \quad  C':=C{T'}^{-1}=:\begin{pmatrix}C_1', & C_2' \end{pmatrix}.
 \end{align*}
Since the block-diagonal structure of $A^\prime$ is preserved by block transformations, we consider
in the following only block-diagonal transformation matrices $T^{\prime\prime}=\diag(T_1'',T_2'')$, $T_1''\in GL_{c}(\R)$, $T_2''\in GL_{N-c}(\R)$ (see Gantmacher \cite{Gantmacher1986}, p.231)
resulting in
\begin{align*}
A''&:=T''A'{T''}^{-1}=\begin{pmatrix}0_{c\times c} & 0_{c\times (N-c)}\\ 0_{ (N-c)\times c} & T_2''A_2'{T_2''}^{-1}\end{pmatrix},\quad
B'':=T''B'=\begin{pmatrix}T_1''B_1'\\  T_2''B_2'\end{pmatrix} \quad \text{ and } \quad\\
C''&:=C'{T''}^{-1}=\begin{pmatrix}C_1'{T_1''}^{-1}, & C_2'{T_2''}^{-1} \end{pmatrix}.
\end{align*}
Due to  Hannan and Deistler \cite[Theorem 2.3.4]{hannandeistler2012}, the triplets $(A,B,C)$, $(A',B',C')$ and $(A'',B'',C'')$ are observationally equivalent and hence, minimal so that due to
\Cref{LemmaMinimality} both $(A_2',B_2',C_2')$ and $(A_2'',B_2'',C_2'')$ are still minimal with $\rank B_1'=\rank B_1''=\rank C_1'=\rank C_1''=c$.
A consequence of \Cref{Lemma:Appendix}(b) is that there exists a transformation matrix $T_1''$ such that $C_1:=C_1'{T_1''}^{-1}$ satisfies {(ii)}. Then $B_1:=T_1''B_1'$.
Since the eigenvalues of $A_2'$ have strictly negative real parts, $(A_2',B_2',C_2')$ forms a stationary
linear state-space model and further, is minimal. Hence, there exists a transformation matrix $T_2''$ such that
$A_2:=T_2''A_2'{T_2''}^{-1}$, $B_2:=T_2''B_2'$  and $C_2:=C_2'{T_2''}^{-1}$ satisfy (iv).
Moreover, the eigenvalues of $A_2'$ and hence, $A_2$ have strictly negative real parts so that (i) is satisfied as well. Finally, we set
$T=T''T'$ and $(A^*,B^*,C^*)=(A'',B'',C'')$.

\textbf{Uniqueness:} Assume that there exist matrices
\begin{align*}
\wt A:=\begin{pmatrix}0_{c\times c} & 0_{c\times (N-c)}\\ 0_{ (N-c)\times c} & \wt A_2\end{pmatrix},\quad
\wt B:=\begin{pmatrix}\wt B_1\\  \wt B_2\end{pmatrix} \quad \text{ and } \quad
\wt C:=\begin{pmatrix}\wt C_1, & \wt C_2 \end{pmatrix},
\end{align*}
so that the state-space model $(\wt A,\wt B,\wt C)$ satisfies the assumptions of this theorem as well. But then due to Hannan and Deistler~\cite[Theorem 2.3.4]{hannandeistler2012} there
 exists a transformation $T\in GL_N(\R)$ with $(\wt A,\wt B,\wt C)=(TA^*T^{-1},TB^*,C^*T^{-1})$. Due to the structure of $A^*$ and $\wt A$
 the equation $TA^*T^{-1}=\wt A$  implies $T=\diag(T_1,T_2)$ with $T_1\in GL_c(\R)$ and $T_2\in GL_{N-c}(\R)$.
To be more precise $\wt A_2=T_2A_2T_2^{-1}$, $\wt B_1=T_1B_1$, $\wt B_2=T_2B_2$, $\wt C_1=C_1T_1^{-1}$ and $\wt C_2=C_2T_2^{-1}$.
Since  $(A_2,B_2,C_2)$ and $(\wt A_2,\wt B_2,\wt C_2)$, respectively are given in canonical form (cf. restriction (iv)), $T_2$
is the identity matrix.
In order to prove uniqueness it remains to show that $T_1$ is the identity matrix as well.
If we now exploit the fact that $C_1$ and $\widetilde{C}_1$ are both positive lower triangular matrices, we further get that $T_1$ is a positive lower triangular matrix  itself.
Due to \Cref{Lemma:Appendix}(a) there exists a unique matrix $U\in M_{d,c}(\R)$ with $U^\mathsf{T}U=I_c$ and a lower positive triangular matrix $L\in GL_c(\R)$
with $C_1=UL$. But then $\wt C_1T_1=C_1=C_1I_c=UL$ with a unique $U$. Thus, $C_1=U=\wt C_1$ and finally, $T_1=I_c$.
\end{proof}


\begin{theorem}
\label{defCointegratedMCARMA}
Let $(A,B,C,L)$ be a  $d$-dimensional minimal state-space model which satisfies $\sigma(A)\subset\{(-\infty,0)+i\R\} \cup \{0\}$, and the algebraic and the geometric multiplicity of the eigenvalue zero is equal to $c$ with $0<c< d$. With the notation of \Cref{ThmCanonFormCoint} define
\begin{align}
    \label{eqRepContCointSSMSepar}
        Y(t)=C_1X_1(t_0) +C_1B_1[L(t)-L(t_0)] +\int_{-\infty}^tC_2\exp(A_2(t-u))B_2\, \dif L(u), \quad t\geq t_0,
\end{align}
and
\begin{align}
\label{eqStatPartofCointSSM}
    Y_{st}(t)=\int_{-\infty}^tC_2\exp(A_2(t-u))B_2\, \dif L(u), \quad t\in\R.
\end{align}
Then $Y$ is a cointegrated solution of the state space model $(A,B,C,L)$ with
   cointegration space spanned by $C_{1}^\perp$ and cointegration rank $r=d-c=\rank C_{1}^\perp$. In particular, $C_{1}^{\perp\,\mathsf{T}} Y=C_{1}^{\perp\mathsf{T}} Y_{st}$ is stationary
   and every component of $C_1^\mathsf{T}Y$ is non-stationary.
Moreover, if $X_1(t_0)$ is independent of $\sigma(L(t)-L(t_0):t\geq t_0)$ then $Y$ is causal.
\end{theorem}
\begin{proof}
First of all, we take  $X_2(t_0)=\int_{-\infty}^{t_0}\exp(A_2(t_0-u))B_2\, \dif L(u)$  in \Cref{ThmCanonFormCoint} which results in the representation of $Y$
in \Cref{ThmCanonFormCoint}. Hence, $Y$ is a solution of the state space model $(A,B,C,L)$.  The process $Y_{st}$ is well-defined and stationary since
$\sigma(A_2)\subseteq \{(-\infty,0)+i\R\}$ and $\E\|L(1)\|^2<\infty$.
Due to $\rank C_1=\rank B_1=c$ (see \Cref{LemmaMinimality}) and $\Sigma_L$  non-singular, we have  $\rank(C_1B_1\Sigma_LB_1^{\mathsf{T}}C_1^{\mathsf{T}})=c$ and finally,
$C_1B_1\Sigma_LB_1^{\mathsf{T}}C_1^{\mathsf{T}}\not=0_{d\times d}$.
Hence, $Y$ as given in \eqref{eqRepContCointSSMSepar}
is indeed an integrated process because the  L\'evy process $(C_1B_1L(t))_{t\geq 0}$ is a non-stationary process with strictly stationary increments.
Moreover, $C_{1}^{\perp\mathsf{T}} Y=C_{1}^{\perp\mathsf{T}} Y_{st}$ is a stationary process.
On the other hand, due to  $C_1^\mathsf{T}C_1=I_c$ by \Cref{ThmCanonFormCoint} (ii) and $\rank(B_1\Sigma_L B_1^{\mathsf{T}})=c$,  every component of $C_1^\mathsf{T}Y=X_1(t_0)+B_1L(t)+C_1^{\mathsf{T}}Y_{st}$ is non-stationary so that  $C_{1}^\perp$ spans the cointegration space with
$\rank C_{1}^\perp=d-c$.
\end{proof}


In a way we have as well the opposite direction of this theorem.

\begin{theorem} \label{Theorem 3.3}
Let $0<c<d$, $C_1\in M_{d,c}(\R), B_1\in M_{c,m}(\R)$ with $\rank C_1=\rank B_1=c$,  $A_2\in M_{N-c,N-c}(\R)$,  $C_2\in M_{d,N-c}(\R)$,  $B_2\in M_{N-c,m}(\R)$, $\sigma(A_2)\subset \{(-\infty,0)+i\R\}$ and $(A_2,B_2,C_2)$ minimal. Furthermore,
\beao
    Y(t)=C_1X_1(t_0) +C_1B_1[L(t)-L(t_0)] +\int_{-\infty}^tC_2\exp(A_2(t-u))B_2\, \dif L(u)), \quad t\geq t_0,
\eeao
 Then $Y$ is a   cointegrated solution  of a minimal linear state-space model $(A,B,C,L)$ with cointegration rank $r=d-c$ satisfying $\sigma(A)\subset\{(-\infty,0)+i\R\} \cup \{0\}$, and the algebraic and the geometric multiplicity of the eigenvalue zero is equal to $c$.
\end{theorem}
\begin{proof} $\mbox{}$\\
Let $(A,B,C)$  be defined as in \Cref{LemmaMinimality}. From \Cref{LemmaMinimality} we also know that $(A,B,C)$ is minimal.
The rest is a conclusion of \Cref{defCointegratedMCARMA}.
\end{proof}

In summary, the following definition is well-defined.

\begin{definition} \label{definition ssp}
Let $0<c<d$, $C_1\in M_{d,c}(\R), B_1\in M_{c,m}(\R)$ ,  $A_2\in M_{N-c,N-c}(\R)$,  $C_2\in M_{d,N-c}(\R)$,  $B_2\in M_{N-c,m}(\R)$, $\sigma(A_2)\subset (-\infty,0)+i\R$ and $0< \rank(C_1B_1)<d$. Then
\beao
    Y(t)=C_1X_1(t_0) +C_1B_1[L(t)-L(t_0)] +\int_{-\infty}^tC_2\exp(A_2(t-u))B_2\, \dif L(u), \quad t\geq t_0,
\eeao
 is called \textbf{cointegrated linear state space process} with parameters $(A_2,B_1,B_2,C_1,C_2,L)$ and cointegration rank $r=\rank(C_1B_1)$.
\end{definition}

Until now, we investigated only solutions of minimal state space models. Finally, we would like to extend it to general state space models.

\begin{lemma} \label{Lemma3.5}
Let $(A,B,C,L)$ be a  $d$-dimensional minimal state-space model with McMillian degree $N$ which satisfies $\sigma(A)\subset\{(-\infty,0)+i\R\} \cup \{0\}$, and the algebraic and the geometric multiplicity of the eigenvalue zero is equal to $c$ with $0<c< d$.
Furthermore, let  $(\mathcal{A},\mathcal{B},\mathcal{C},L)$ be a state-space model of dimension $\mathcal{N}$. Assume
\beao
    C(zI_N-A)^{-1}B=\mathcal{C}(zI_{\mathcal{N}}-\mathcal{A})^{-1}\mathcal{B}.
\eeao
Then there exists a cointegrated state space process $Y$ which is a solution of both $(A,B.C)$ as well as $(\mathcal{A},\mathcal{B},\mathcal{C})$.
\end{lemma}

 From this we see that the transfer function uniquely defines cointegrated state space processes. Moreover, it means that if we have a state model $(A,B,C)$
 and if we know a minimal form we can calculate  a cointegrated solution.

\begin{proof}
 With the notation of \Cref{ThmCanonFormCoint} define
  \begin{align*}
A^*:=\begin{pmatrix}0_{c\times c} & 0_{c\times (N-c)}\\ 0_{ (N-c)\times c} & A_2\end{pmatrix}\in M_{N,N}(\R),\,
B^*:=\begin{pmatrix}B_1\\  B_2\end{pmatrix}\in M_{N, m}(\R) \text{ and }
C^*:=\begin{pmatrix}C_1, & C_2 \end{pmatrix}\in M_{d, N}(\R)
\end{align*}
and let $T\in GL_N(\R)$ be the transformation matrix satisfying $(A,B,C)=(TA^*T^{-1},TB^*,C^*T^{-1})$.
Furthermore, let $\mathcal{T}\in GL_{\mathcal{N}}(\R)$ be the transformation bringing $\mathcal{A}$ in Jordan normal form.
Define $\mathcal{A}^*=\mathcal{T}\mathcal{A}\mathcal{T}^{-1}=\diag(\mathcal{A}^*_1,\mathcal{A}_2^*,\mathcal{A}_3^*)$ where $\mathcal{A}^*_1\in M_{k_1,k_1}$
is the Jordan block
of the eigenvalue $0$, $\mathcal{A}^*_2\in M_{k_2,k_2}$ is the Jordan block
of eigenvalues with strictly negative real parts and $\mathcal{A}^*_3\in M_{k_3,k_3}$ is the Jordan block
of eigenvalues with positive real parts with $k_1,k_2,k_3\in\N_0$ where $k_1+k_2+k_3=\mathcal{N}$. Further define $(\mathcal{A}^*,\mathcal{B}^*,\mathcal{C}^*)=(\mathcal{T}\mathcal{A}\mathcal{T}^{-1},\mathcal{T}\mathcal{B},\mathcal{C}\mathcal{T}^{-1})$
and write accordingly
\beao
\mathcal{B}^*=\begin{pmatrix}\mathcal{B}_1\\  \mathcal{B}_2 \\ \mathcal{B}_3\end{pmatrix}\in M_{\mathcal{N}, m}(\R) \quad \text{and} \quad
\mathcal{C}^*=\begin{pmatrix}\mathcal{C}_1, & \mathcal{C}_2, &\mathcal{C}_3 \end{pmatrix}\in M_{d, \mathcal{N}}(\R).
\eeao
Hence, we receive
\beao
    \frac{1}{z}C_1B_1+C_2(zI_{N-c}-A_2)^{-1}B_2&=&C(zI_N-A)^{-1}B= \mathcal{C}(zI_{\mathcal{N}}-\mathcal{A})^{-1}\mathcal{B}\\
    &=&\mathcal{C}_1(zI_{k_1}-\mathcal{A}_1)^{-1}\mathcal{B}_1+\mathcal{C}_2(zI_{k_2}-\mathcal{A}_2)^{-1}\mathcal{B}_2+\mathcal{C}_3(zI_{k_3}-\mathcal{A}_3)^{-1}\mathcal{B}_3.
\eeao
However, this is only possible if 
\beam \label{eq1}
    C_1B_1= \mathcal{C}_1\mathcal{B}_1, \quad C_2(zI_{N-c}-A_2)^{-1}B_2
        =\mathcal{C}_2(zI_{k_2}-\mathcal{A}_2)^{-1}\mathcal{B}_2 \mbox{ and }  0_{d\times m}=\mathcal{C}_3(zI_{k_3}-\mathcal{A}_3)^{-1}\mathcal{B}_3
\eeam
(see also Bernstein~\cite[Theorem 12.9.16]{Bernstein2009}).
Let $\Gamma$ be a closed contour in $\C$ winding around each eigenvalue of $\mathcal{A}_2$ exactly once, then by the
spectral representation theorem \cite[Theorem 17.5]{lax2002}
\beao
    \e^{\mathcal{A}_2t}=\frac{1}{2\pi i}\int_\Gamma\e^{zt}(zI_{k_2}-\mathcal{A}_2)^{-1}\,dz
\eeao
and hence,
\beam \label{eq2}
    \mathcal{C}_2\e^{\mathcal{A}_2t}\mathcal{B}_2&=&\frac{1}{2\pi i}\int_\Gamma\e^{zt} \mathcal{C}_2(zI_{k_2}-\mathcal{A}_2)^{-1}\mathcal{B}_2\,dz\nonumber\\
        &=&\frac{1}{2\pi i}\int_\Gamma\e^{zt} {C}_2(zI_{{N-c}}-{A_2})^{-1}{B}_2\,dz= C_2\e^{A_2t}B_2.
\eeam
Then using \eqref{eq1} and \eqref{eq2} gives
\begin{align}
    Y(t)&:=C_1X_1(t_0) +C_1B_1[L(t)-L(t_0)] +\int_{-\infty}^tC_2\exp(A_2(t-u))B_2\, \dif L(u) \nonumber\\
        &=C_1X_1(t_0) +\mathcal{C}_1\mathcal{B}_1[L(t)-L(t_0)] +\int_{-\infty}^t\mathcal{C}_2\exp(\mathcal{A}_2(t-u))\mathcal{B}_2\, \dif L(u)\nonumber\\
        &=\mathcal{C}^*\exp(\mathcal{A}^*(t-t_0))\mathcal{X}^*(t_0) +\int_{t_0}^t\mathcal{C}^*\exp(\mathcal{A}^*(t-u))\mathcal{B}^*\, \dif L(u) \label{eq4}
\intertext{with $\mathcal{X}^*(t_0)=(\mathcal{X}_1^*(t_0)^{\mathsf{T}}\,\,\mathcal{X}_2^*(t_0)^{\mathsf{T}}\,\,\,\mathcal{X}_3^*(t_0)^{\mathsf{T}})^{\mathsf{T}}$ and $$C_1X_1(t_0)=\mathcal{C}_1\mathcal{X}_1^*(t_0), \quad \mathcal{X}_2^*(t_0)=\int_{-\infty}^{t_0}\exp(\mathcal{A}_2(t_0-u))\mathcal{B}_2\, \dif L(u), \quad \mathcal{X}_3^*(t_0)=0_{k_3}. \quad$$  Remark that $C_1X_1(t_0)=\mathcal{C}_1\mathcal{X}_1^*(t_0)$ is possible since due to $\rank C_1=c=\rank(C_1B_1)$ and \eqref{eq1}
we have
\beao
    \text{im}(C_1)=\text{im}(C_1B_1)=\text{im}(\mathcal{C}_1\mathcal{B}_1)\subseteq \text{im}(\mathcal{C}_1)
\eeao
where $\text{im}(C_1)$ denotes the image of $C_1$. Finally, plugging in  $\mathcal{X}(t_0)=\mathcal{T}^{-1}\mathcal{X}^*(t_0)$ and $(\mathcal{A}^*,\mathcal{B}^*,\mathcal{C}^*)=(\mathcal{T}\mathcal{A}\mathcal{T}^{-1},\mathcal{T}\mathcal{B},\mathcal{C}\mathcal{T}^{-1})$ in \eqref{eq4} result in}
        Y(t)&=\mathcal{C}\exp(\mathcal{A}(t-t_0))\mathcal{X}(t_0) +\int_{t_0}^t\mathcal{C}\exp(\mathcal{A}(t-u))\mathcal{B}\, \dif L(u) \nonumber
\end{align}
which results in the statement.
\end{proof}

\subsection{Probabilistic properties of cointegrated state space processes}

 We assume for reasons of simplicity that $t_0=0$ for the rest of the paper.
Due to the decoupled canonical form  of the cointegrated state-space process the covariance matrix can also be decomposed.

\begin{lemma}
\label{propCovCointSSM}
Let $Y$ be a cointegrated linear state-space process with parameters $(A_2,B_1,B_2,C_1,C_2,L)$ and $X_1(0)$ be independent of $\sigma(L(t):t\geq 0)$ with $\E\|X_1(0)\|^2<\infty$.
Then
$\E[Y(t)]=C_1\E[X_1(0)]$ for $t\geq 0$ and
 for $t,s\geq 0$ we have
\begin{eqnarray*}
\label{eqCovCointSSM}
\Cov(Y(t),Y(t+s))&=&C_1\Cov(X_1(0))C_1^{\mathsf{T}}+C_2\exp(A_2\, s)\Gamma_0C_2^\mathsf{T}+\int_0^tC_2\exp(A_2u)B_2\Sigma_L(C_1B_1)^\mathsf{T}\,\dif u\\
&& +\int_s^{t+s}C_1B_1\Sigma_LB_2^\mathsf{T}\exp(A_2^\mathsf{T}u)C_2^\mathsf{T}\,\dif u+t\cdot C_1B_1\Sigma_L(C_1B_1)^\mathsf{T},
\end{eqnarray*}
where
\begin{eqnarray*}
\label{eqCovStatPartCointSSM}
\Gamma_0=\int_0^\infty\exp(A_2u)B_2\Sigma_L B_2^\mathsf{T}\exp(A_2^\mathsf{T}u)\,\dif u.
\end{eqnarray*}
\end{lemma}
\begin{proof}
We obtain for the expectation evidently
\begin{align*}
\E[Y(t)]&=\E\left[C_1X_1(0) +C_1B_1L(t) +C_2\int_{-\infty}^t\exp(A_2(t-u))B_2\, \dif L(u)\right]
=C_1\E[X_1(0)].
\end{align*}
W.l.o.g. we assume $X_1(0)=0_c$. Then
\begin{align*}
&\Cov(Y(t),Y(t+s))=\E\left[Y(t)Y(t+s)^\mathsf{T}\right]
\\
&=C_1B_1\E\left[L(t)L(t+s)^\mathsf{T}\right]\left(C_1B_1\right)^\mathsf{T}+C_1B_1\E\left[\int_0^t 1\, \dif L(u)\left(\int_{-\infty}^{t+s}\exp(A_2(t+s-u))B_2\, \dif L(u)\right)^\mathsf{T}\right]C_2^\mathsf{T}
\\
&\quad +C_2\E\left[\int_{-\infty}^{t}\exp(A_2(t-u))B_2\, \dif L(u)\left(\int_0^{t+s} 1\, \dif L(u)\right)^\mathsf{T}\right]\left[C_1B_1\right]^\mathsf{T}+\E[Y_{st}(t)Y_{st}(t+s)^\mathsf{T}]
\\
&=t\cdot C_1B_1\Sigma_L\left(C_1B_1\right)^\mathsf{T}+C_1B_1\E\left[\int_0^t 1\, \dif L(v)\left(\int_{0}^{t}\exp(A_2(t+s-u))B_2\, \dif L(u)\right)^\mathsf{T}\right]C_2^\mathsf{T}
\\
&\quad+C_2\E\left[\int_{0}^{t}\exp(A_2(t-u))B_2\, \dif L(u)\left(\int_0^{t} 1\, \dif L(v)\right)^\mathsf{T}\right]\left(C_1B_1\right)^\mathsf{T}+\E[Y_{st}(t)Y_{st}(t+s)^\mathsf{T}],
\end{align*}
and finally, the result follows by calculating all the remaining expectations and Marquardt and Stelzer \cite[Proposition 3.13]{MarquardtStelzer2007}.
\end{proof}
The time dependence of the covariance function is clearly visible in this representation and hence, this process is indeed non-stationary.

\begin{lemma}
Let $Y$ be the cointegrated linear state-space process and $X_1(0)$ be independent of $\sigma(L(t):t\geq 0)$ with $\E\|X_1(0)\|^r<\infty$.
If $\E\|L(1)\|^{r}<\infty$ for some $r\in\N$  then $\E\|Y(t)\|^r<\infty$ for any $t\geq 0$.
\end{lemma}
\begin{proof}
The existence of the $r^{th}$-moment follows immediately by Marquardt and Stelzer \cite[Proposition 3.30]{MarquardtStelzer2007} giving $\E\|Y_{st}(t)\|^{r}<\infty$
for all $t\geq 0$.
\end{proof}

\section{Cointegrated MCARMA processes} \label{Section:3}
\subsection{MCARMA processes and state space models}

As already mentioned in the introduction, the motivation for the definition of an MCARMA process
is beeing the solution of the $d$-dimensional stochastic differential equation
\begin{subequations} \label{MCARMA}
\begin{eqnarray} \label{eq1.1}
     {P}(D)Y(t)={Q}(D)D L(t) \quad \mbox{ for } t\geq t_0,
\end{eqnarray}
where $D$ is the differential operator with respect to $t$,
\begin{align}
\label{PolP}
    {P}(z)&:= I_{d\times d}z^p+P_1z^{p-1}+\ldots+P_{p-1}z+P_p
\intertext{is the autoregressive polynomial with $P_1,\ldots, P_p\in M_{d,d}(\R)$ and}
\label{PolQ}
    {Q}(z)&:=Q_0z^q+ Q_1z^{q-1}+\ldots+ Q_{q-1}z+Q_q
\end{align}
\end{subequations}
is the moving average polynomial with $Q_0,\ldots,Q_q\in M_{d,m}(\R)$. Since a Lévy process is not differentiable
we present the formal definition of an MCARMA process now.

\begin{definition} \label{definition CARMA}
Let $P(z)$ and $Q(z)$ be defined as in \eqref{MCARMA} with $p>q$. Moreover,
\beam \label{Adefinition CARMA}
    A_M:=\left(\begin{array}{ccccc}
        0_{d\times d} & I_{d} & 0_{d\times d} & \cdots & 0_{d\times d}\\
        0_{d\times d} & 0_{d\times d} & I_{d} & \ddots & \vdots \\
        \vdots & & \ddots & \ddots & 0_{d\times d}\\
        0_{d\times d} & \cdots & \cdots & 0_{d\times d} & I_{d}\\
        -P_p & -P_{p-1} & \cdots & \cdots & -P_1
    \end{array}\right) \in \R^{pd\times pd},
\eeam
$C_M:=(I_{d},0_{d\times d},\ldots,0_{d\times d}) \in \R^{d\times pd}$ and $B_M:=(\beta_1^\mathsf{T} \cdots \beta_p^\mathsf{T})^\mathsf{T}\in \R^{pd\times m}$ with
\beao
   \beta_1:=\ldots:=\beta_{p-q-1}:=0_{d\times m}\quad
   \mbox{ and } \quad \beta_{p-j}:=-\sum_{i=1}^{p-j-1}P_i \beta_{p-j-i}+Q_{q-j},  \quad j=0,\ldots,q.
\eeao
Then the $\R^{d}$-valued \textbf{ MCARMA$(p,q)$ process} $Y=(Y(t))_{t\geq 0}$ is  defined by the state-space
equation
\begin{subequations} \label{CARMA:observation}
\beam
    Y(t)=C_M X(t)
\eeam
where
$X=(X(t))_{t\geq 0}$ is the solution to the $(pd)$-dimensional
stochastic differential equation
\begin{equation} \label{cstateeq}
\dd X(t)=A_M X(t)\,\dd t+B_M \,\dd L(t).
\end{equation}
\end{subequations}
\end{definition}

\begin{definition}
Let the cointegrated state space process $Y$ be the solution of the state space model $(A_M,B_M,C_M)$ as given in \eqref{Adefinition CARMA}.
Then $Y$ is called \textbf{cointegrated MCARMA process}.
\end{definition}

\begin{remark} $\mbox{}$
\begin{itemize}
    \item[(i)] In particular, MCARMA$(1,0)$ processes and the process $X$ in \eqref{cstateeq} are multivariate Ornstein-Uhlenbeck processes.
    \item[(ii)] 
    From Marquardt and Stelzer \cite[Lemma 3.8]{MarquardtStelzer2007}
    we already know that $0\in\sigma(A_M)$ iff $0\in\sigma(P_p)$. Thus, the matrix $P_p$ plays a crucial role in the characterization
    of cointegration of MCARMA processes.
    \item[(iii)] In the case of stationary MCARMA processes Marquardt and Stelzer \cite{MarquardtStelzer2007}
    give details on the well-definedness of this model and that the MCARMA process can be interpreted as the
    solution of \eqref{MCARMA}. This holds as well for cointegrated MCARMA processes and can be seen by rewriting
        \eqref{CARMA:observation}  line by line.
\end{itemize}
\end{remark}
Schlemm and Stelzer \cite[Corollary 3.4]{SchlemmStelzer2012} prove that the class of stationary MCARMA processes
    and the class of stationary state-space processes are equivalent. An analog result holds for the cointegrated case.

\begin{proposition}
        A $d$-dimensional process $Y$ is a cointegrated state space process  iff $Y$ is  a cointegrated MCARMA process.
\end{proposition}
\begin{proof}
  By definition, every cointegrated MCARMA process is a cointegrated state-space process.
 Conversely, let $(A,B,C)$ be a minimal state space model with McMillian degree $N$
   and the cointegrated state space process $Y$  be a solution of $(A,B,C)$. Then due to  Caines~\cite[Theorem 8 in Appendix 2]{Caines}
    (cf. Bernstein~\cite[Proposition 4.7.16]{Bernstein2009}) the matrix function $C(zI_N-A)B$ has a left coprime
   polynomial fraction description. Hence, there exists polynomials $P(z),Q(z)$ as  in \eqref{MCARMA} of some order $p$, respectively $p-1$ with
    $
        C(zI_N-A)^{-1}B=P(z)^{-1}Q(z).
    $
    Define $(A_M,B_M,C_M)$  as in \eqref{Adefinition CARMA} using this polynomials $P(z)$ and $Q(z)$. Then $$C(zI_N-A)^{-1}B=P(z)^{-1}Q(z)=C_M(zI_N-A_M)^{-1}B_M$$
    (cf. proof of \cite[Theorem 3.3]{SchlemmStelzer2012}) and hence, due to
\Cref{Lemma3.5}, the process $Y$ is as well a solution of $(A_M,B_M,C_M)$. This means $Y$ is a cointegrated MCARMA process.
\end{proof}


%

%
%
\subsection{Characterization of cointegrated MCARMA processes}
%

In this section, we characterize cointegration with respect to the matrices $P_p$ and $P_{p-1}$ in the autoregressive polynomial. The next
proposition is an extension of Comte \cite[Proposition 7]{Comte1999} for MCAR processes to MCARMA processes.

\begin{proposition}
\label{thmCoIntMARep}
Let $Y$ be a  MCARMA$(p,q)$ process as a solution of \eqref{MCARMA} with $p>q+1$.
Furthermore, we assume:
\begin{itemize}
\item[(M1)]
If $\det (P(z))=0$  then either $\Re(z)<0$  or  $z=0$.
\item[(M2)]
$rank(P_p)=rank(P(0))=d-c$, $0<c<d$, and $P_p=\alpha\beta^\mathsf{T}$, where the adjustment matrix $\alpha\in M_{d,d-c}(\R)$ and the cointegration matrix $\beta\in M_{d,d-c}(\R)$
   have full rank $r=d-c$.
\item[(M3)]
$\alpha^{\perp\mathsf{T}}
    P_{p-1}\beta^\perp\in M_{c,c}(\R)$  with full rank $c$.
\end{itemize}
Then  $DY$ and $\beta^\mathsf{T}Y$ are stationary processes where $DY$ is  the differential of $Y$ in the $L^2$ sense.
\end{proposition}

An important representation of cointegrated processes in discrete time is the so-called error correction form, which was first introduced by Sargan \cite{Sargan1964wages}.
For an MCARMA process an analog version of the error correction form is
\begin{align}
\label{eqErrorCorrectionForm1}
P^\ast(D)DY(t)=&-P_pY(t)+Q(D)DL(t),
\end{align}
where the polynomial $P^\ast$ has the representation
$P^\ast(z):=\frac{P(z)-P_p}{z}$.
If the first-order mean square derivative $DY$ of $Y$ exists then $DY$ is stationary
iff  $Y$ has stationary increments (see the proof of Comte~\cite[Proposition 1]{Comte1999}).

\begin{proof}[Proof of \Cref{thmCoIntMARep}]
By multiplying \eqref{eqErrorCorrectionForm1} by $\alpha$ and $\alpha^{\perp\mathsf{T}}$ we obtain
with $P_p=\alpha\beta^\mathsf{T}$ and $\alpha^{\perp\mathsf{T}}\alpha=0_{c\times (d-c)}$ the following
equations
\begin{equation} \label{eqProofCoIntMARep1}
\begin{split}
\alpha^\mathsf{T}Q(D)DL(t)&=\alpha^\mathsf{T}\alpha\beta^\mathsf{T} Y(t)+\alpha^\mathsf{T}P^\ast(D)DY(t),\\
\alpha^{\perp\mathsf{T}} Q(D)DL(t)&=\alpha^{\perp\mathsf{T}} P^\ast(D) DY(t).
\end{split}
\end{equation}
Since the system \eqref{eqProofCoIntMARep1} is not invertible in $Y$ and $DY$ we define new processes
\begin{align*}
Z(t):=(\beta^\mathsf{T}\beta)^{-1}\beta^\mathsf{T} Y(t)\quad\text{and}\quad
V(t):=(\beta^{\perp\mathsf{T}}\beta^\perp)^{-1}\beta^{\perp\mathsf{T}} DY(t) \quad \text{for } t\geq 0,
\end{align*}
 and obtain thereby invertibility.
The matrix $R:=(\beta,\beta^\perp)\in M_{d, d}(\R)$ of rank $d$ satisfies
\begin{align}
\label{eqProjMatrR}
R(R^\mathsf{T}
R)^{-1}R^\mathsf{T}=\beta(\beta^{\mathsf{T}}
\beta)^{-1}\beta^\mathsf{T}+\beta^\perp(\beta^{\perp\mathsf{T}}
\beta^\perp)^{-1}\beta^{\perp\mathsf{T}}=I_{d}
\end{align}
since it is the sum of the projection matrices on the range and the projection matrix on the null space of $\beta$.
 Moreover, for
$\bar{\beta}:=\beta(\beta^\mathsf{T}\beta)^{-1}\in M_{d,d-c}(\R)$ and
$\bar{\beta}^\perp:=\beta^\perp(\beta^{\perp\mathsf{T}}\beta^\perp)^{-1}\in M_{d, c}(\R)$ we have due to \eqref{eqProjMatrR} that
$\beta\bar{\beta}^\mathsf{T}+\beta^\perp\bar{\beta}^{\perp\mathsf{T}}=I_d$ holds. Furthermore, we have
\begin{align*}
DY(t)=(\beta\bar{\beta}^\mathsf{T}+\beta^\perp\bar{\beta}^{\perp\mathsf{T}})DY(t)=\beta
DZ(t)+\beta^\perp V(t).
\end{align*}
Rewriting system \eqref{eqProofCoIntMARep1} with the newly defined variables yields
\begin{align*}
\alpha^\mathsf{T}Q(D)DL(t)&=\alpha^\mathsf{T}\alpha(\beta^\mathsf{T}\beta) Z(t)+\alpha^\mathsf{T}P^\ast(D)\beta
DZ(t)+\alpha^\mathsf{T}P^\ast(D)\beta^\perp V(t),
\\
 \alpha^{\perp\mathsf{T}} Q(D)DL(t)&=\alpha^{\perp\mathsf{T}} P^\ast(D) \beta
 DZ(t)+\alpha^{\perp\mathsf{T}}P^\ast(D)\beta^\perp V(t).
\end{align*}
Rearranging the last expressions leads to
\begin{eqnarray}
\label{eqStationarySystem}
\widetilde{P}(D)(Z(t)^\mathsf{T},V(t)^\mathsf{T})^\mathsf{T}=(\alpha,\alpha^\perp)^\mathsf{T}Q(D)DL(t),
\end{eqnarray}
where the matrix polynomial $\widetilde{P}(z)$ is given by
\begin{align}
\label{eqwidetildP}
\widetilde{P}(z):=\begin{pmatrix}
\alpha^\mathsf{T}\alpha(\beta^\mathsf{T}\beta)+\alpha^\mathsf{T}P^\ast(z)\beta z & \alpha^\mathsf{T}P^\ast(z)\beta^\perp
\\
\alpha^{\perp\mathsf{T}} P^\ast(z) \beta z & \alpha^{\perp\mathsf{T}}P^\ast(z)\beta^\perp
\end{pmatrix}.
\end{align}
By assumption (M2) and (M3) we have
\begin{align*}
&\det(\widetilde{P}(0))=\det\begin{pmatrix}
\alpha^\mathsf{T}\alpha(\beta^\mathsf{T}\beta)& \alpha^\mathsf{T}P^\ast(0)\beta^\perp
\\
0_{c\times d-c} & \alpha^{\perp\mathsf{T}}P^\ast(0)\beta^\perp
\end{pmatrix}
=\det(\alpha^\mathsf{T}\alpha)\det(\beta^\mathsf{T}\beta)\det(\alpha^{\perp\mathsf{T}}P^\ast(0)\beta^\perp)\not=0,
\end{align*}
where all matrices in the last line have full rank and consequently, a non-zero determinant.
Then we can see due to \eqref{eqErrorCorrectionForm1} and \eqref{eqwidetildP} that $\widetilde{P}(z)=(\alpha,\alpha^\perp)^\mathsf{T} P(z)
(\beta,\beta^\perp/z)$ for $z\not=0$ and thus,
\begin{align} \label{P tilde}
\det\big(\widetilde{P}(z)\big)=\frac{1}{z^{d-r}}\det(\alpha,\alpha^\perp)^\mathsf{T}\det
\left(P(z)\right)\det(\beta,\beta^\perp)\not=0.
\end{align}
Finally,  $\widetilde{P}(z)$ has the same roots as $P(z)$, except the null ones and the non-zero roots are
assumed to have negative real parts due to (M1). Thus,
 \eqref{eqStationarySystem} has  a stationary solution $(Z,V)$ due to Marquard and Stelzer~\cite{MarquardtStelzer2007}. Then the  processes $DY(t)=\beta
DZ(t)+\beta^\perp V(t)$ and  $\beta^\mathsf{T}Y(t)=(\beta^\mathsf{T}\beta)Z(t)$
are stationary as well as linear combinations of stationary processes.
\end{proof}
\begin{theorem} \label{Theorem 4.6}
Let $Y$ be a  MCARMA$(p,q)$ process satisfying the assumptions of \Cref{thmCoIntMARep}. Assume additionally
\begin{itemize}
    \item[(M4)] $P(z)$ and $Q(z)$ are left coprime.
\end{itemize}
Then $Y$ is cointegrated.
\end{theorem}
\begin{proof}
From the proof of Schlemm and Stelzer~\cite[Theorem 3.3]{SchlemmStelzer2012} we already know that
\beao
    G(z):=P^{-1}(z)Q(z)=C_M(I_Nz-A_M)^{-1}B_M
\eeao
with $(A_M,B_M,C_M)$ as in \eqref{Adefinition CARMA}, and due to (M2)  that $\det(P(0))=0$. Then a conclusion of (M4) and Bernstein~\cite[Proposition 4.7.16]{Bernstein2009} is that
$0$ is a pole of $G$. Let $(A,B,C)$ be a minimal representation of $G$. Then as well $0\in\sigma(A)$
by Bernstein~\cite[Proposition 12.9.16]{Bernstein2009}. It remains to show that the algebraic and the geometric multiplicity of the
eigenvalue $0$ is the same.
 On the one hand, the representation of  $A_M$ in \eqref{Adefinition CARMA} and the proof of \Cref{thmCoIntMARep} yield
\beao
    \det(I_Nz-A_M)=\det(P(z))=Cz^{c}\det(\widetilde P(z))
\eeao
where $C\not=0$ is a constant and $\det(\widetilde P(0))\not=0$. On the other hand, $$N-\rank A_M=d-\rank P_p=d-(d-c)=c$$ due to (M2).
Hence, the algebraic and the geometric multiplicity
of the eigenvalue $0$ of $A_M$ is $c$. Since $(A,B,C)$ is minimal  the algebraic and the geometric multiplicity of the eigenvalue 0 of $A$ is then as well equal to some $\widetilde c$
with $0<\widetilde c\leq c<d$. Finally,
 $\sigma(A)\subseteq\sigma(A_M)\subseteq\{(-\infty,0)+i\R\}\cup\{0\}$ by (M1) and Bernstein~\cite[Proposition 12.9.16]{Bernstein2009}. Thus,   $Y$  is non-stationary by \Cref{ThmCanonFormCoint},
 and  $Y$ has stationary increments with $\beta^{\mathsf{T}}Y$ stationary by \Cref{thmCoIntMARep}.
\end{proof}
\Cref{Theorem 4.6} for an MCARMA processes  complements \Cref{defCointegratedMCARMA} for a state space model.
Conditions (M1)-(M4) imply that the  geometric and the algebraic multiplicity of the eigenvalue 0 is $c$.
Finally, we make some remarks on the last result and its implications on cointegration for MCARMA models.

\begin{remark}
The assumptions in \Cref{Theorem 4.6} have the following relevance under (M4) which replaces the assumption of minimality:
\begin{itemize}
\item Assumption (M1) guarantees that the process is non-stationary.
\item Assumption (M2) guarantees that there exist linear combinations which are stationary.
\item Assumption (M3) guarantees that the process is integrated of order one and not of higher order.
\item Assumption (M4) guarantees that $0$ is not a zero of $Q(z)$.
\end{itemize}
If the cointegration rank is zero, i.e. $P_p=0_{d\times d}$, we have no cointegration vector and thus, the process is not cointegrated. However, the process is integrated. On the other hand, if the rank of $P_p$ is equal to $d$, i.e. $P_p$ is of full rank, the process is stationary. This means that all eigenvalues have strictly negative real parts and (M1) is satisfied.  Additionally, (M3) is automatically satisfied, whereas (M2) is violated. Therefore cointegration arises when the rank of $P_p$ satisfies $0<r<d$. Hence, it depends
essentially on the matrix $P_p$ if we have a stationary, integrated or even cointegrated MCARMA process.
\end{remark}

\begin{remark}
In this subsection we described conintegrated MCARMA processes which are in particular integrated. However, there are further possibilities to define integrated MCARMA processes.
We present two natural ways.
\begin{enumerate}
\item The first method starts with a stationary $d$-dimensional MCARMA$(p,q)$ process $Y$. The integrated process is defined by integration of $Y$, namely $I(t)=\int_0^tY(s)\,\dif s$. Assume that the process $Y$ satisfies $P(D)Y(t)=Q(D)DL(t)$ and define $P^\ast(z):=zP(z)$. Then the differential equation for the integrated process is
    $$P^*(D)I(t)=P(D)DI(t)=Q(D)DL(t).$$
     The order of the polynomial $P^\ast(z)$ is $p^\ast:=p+1$. Obviously $(I(t))_{t\geq0}$ is then an MCARMA process with parameters $(p^\ast,q)$ and $p^\ast>q$  as well.
\item The second method uses a non-stationary $d$-dimensional MCARMA$(p,q)$ process $I:=(I(t))_{t\geq0}$ where $DI$ is stationary. Assume, that the process $I$ satisfies $P(D)I(t)=Q(D)DL(t)$, $t\geq 0$, and define $Q^\ast(z):=zQ(z)$. Then we have
    $$P(D)DI(t)=D[P(D)I(t)]=D[Q(D)DL(t)]=Q^\ast(D)DL(t).$$
    Clearly, $DI(t)=:Y(t)$ is an MCARMA$(p,q+1)$ process. Again, this implies that we need the assumption $p>q+1$.
\end{enumerate}
The different definitions of the integrated MCARMA process $I$ are not equivalent. Both have in common that $DI$ is stationary and $I$ is an MCARMA process, whereas in the first definition (i) there exist no $\beta$ so that $\beta^\mathsf{T} I$ is stationary implying that $I$ is not cointegrated. In  (ii), $P_p$ is not fixed to be zero, thus we allow the process to be cointegrated.
\end{remark}

\section{Cointegrated state-space processes sampled equidistantly}
\label{sec:KalmanFilterforCointegratedProcesses}

In order to estimate the model parameters  of the state-space model $Y$ from observations in discrete time, the probabilistic properties of sampled versions of
cointegrated state-space processes are of special interest. Therefore, we investigate in this section cointegrated state-space processes  sampled
at an equidistant time-grid with distance $h>0$. The cointegration property of the continuous-time process $Y$ transfers directly
 to its sampled version $Y^{(h)}=(Y(nh))_{n\in\N_0}$ by Comte \cite[Proposition 3]{Comte1999}.
Based on the results presented here, Fasen and Scholz~\cite{Fasen:Scholz:2016b} develop a quasi-maximum-likelihood method to estimate
the parameters of a cointegrated state space process by observations $Y(h),\ldots,Y(nh)$.
Before we start we present some assumptions of this section.

\begin{assumptionletter} \label{AssumptionB}
Let $0<c<d$ and
\beao
    Y(t)=C_1X_1(t_0) +C_1B_1L(t) +\int_{-\infty}^tC_2\exp(A_2(t-u))B_2\, \dif L(u)), \quad t\geq 0,
\eeao
be a cointegrated state process with $C_1\in M_{d,c}(\R), B_1\in M_{c,m}(\R)$, $A_2\in M_{N-c,N-c}(\R)$,  $C_2\in M_{d,N-c}(\R)$ and  $B_2\in M_{N-c,m}(\R)$, Furthermore,
suppose the following conditions are satisfied:
\begin{enumerate}[(B1)]
    \item $\rank C_1=\rank B_1=c$.
    \item $\sigma(A_2)\subset \{(-\infty,0)+i\R\}$ and $(A_2,B_2,C_2)$ is minimal.
    \item $(A,B,C)$ is defined as in \Cref{LemmaMinimality}.
    \item \label{B5} $\rank C=d$.
    \item $X_1(0)$ is independent of $\sigma(L(t):t\geq 0)$ and $\E\|X_1(0)\|^2<\infty$.
    \item For any $\lambda,\lambda'\in\sigma(A)=\sigma(A_2)\cup\{0\}$ and any $k\in\Z\backslash\{0\}$: $\lambda-\lambda'\not=2\pi k/h$ (Kalman-Bertram criterion).
 \end{enumerate}
\end{assumptionletter}
Throughout the rest of the paper we assume that \autoref{ALevy} and \ref{AssumptionB} always  hold.

\subsection{Probabilistic properties of a cointegrated state space processes sampled equidistantly} \label{Section 5.1}

First of all,  we derive a state space representation of the discrete-time process $Y^{(h)}$.
Therefore let us define the sequence
\begin{eqnarray}
\label{eqNoiseTermR}
    \xi_n^{(h)}:=\int_{(n-1)h}^{nh}\exp(A(nh-u))B\,\dif L(u)=\begin{pmatrix}B_1 \Delta L^{(h)}_n \\\int_{(n-1)h}^{nh}\exp({A_2(nh-u)})B_2\,\dif L(u) \end{pmatrix}
    =:\begin{pmatrix} \xi_{1,n}^{(h)} \\ \xi_{st,n}^{(h)} \end{pmatrix},
\end{eqnarray}
where $\Delta L^{(h)}_n :=L(nh)-L((n-1)h)$, $n\in\Z$.
The sequence $\xi^{(h)}=(\xi_n^{(h)})_{n\in\Z}$ is  i.i.d.  with mean zero and covariance matrix
\begin{eqnarray*}
\label{eqCovMatrixIIDseqsampledCoint}
{\Sigma}^{(h)}_\xi:=
\E( \xi_n^{(h)}\xi_n^{(h)\mathsf{T}})=\int_0^h\begin{pmatrix} B_1\Sigma_LB_1^\mathsf{T} & \e^{A_2u}B_2\Sigma_LB_1^\mathsf{T}
\\
B_1\Sigma_LB_2^\mathsf{T}\e^{A_2^\mathsf{T}u} & \e^{A_2u}B_2\Sigma_LB_2^\mathsf{T}\e^{A_2^\mathsf{T}u}\end{pmatrix}\,\dif u.
\end{eqnarray*}
The matrix $\Sigma^{(h)}_\xi$ is positive definite due to Bernstein~\cite[Theorem 12.6.18]{Bernstein2009}, $\Sigma_L$ positive definite and due to the controllability of $(A,B,C)$ implied by the minimality of $(A,B,C)$  given in \Cref{LemmaMinimality}.
Define $\xi_n^*=(\Sigma_\xi^{(h)})^{1/2}\xi_n^{(h)}$, $n\in\Z$. Then $(\xi_n^*)_{n\in\Z}$ is an iid sequence with $\E(\xi_n^*)=0_N$ and $\Var(\xi_n^*)=I_N$.
Due to \eqref{defObservEqCoint} the discrete-time process $Y^{(h)}$ is then a solution of the discrete-time state space model
\begin{align} \label{Y discrete state}
\begin{array}{rl}
X_n^{(h)}
&=\exp(Ah)X_{n-1}^{(h)}+(\Sigma_\xi^{(h)})^{1/2}\xi_n^*,\\
Y^{(h)}_n&=C X_{n}^{(h)}.
\end{array}
\end{align}
 From this representation we see clearly the connection of the cointegration property of the continuous-time process $Y$
  with the discrete-time process $Y^{(h)}$. The continuous-time model has transition matrix $A$ and the discrete-time model has transition matrix $\exp(A)$. But
   $\exp(A)$ has eigenvalues equal to one (a unit root) iff $A$ has eigenvalues equal to zero.

\begin{lemma} \label{lemma 5.1a}
The discrete-time state space model $(\exp(A),\Sigma_\xi^{1/2}C,\xi^*)$ as given in \eqref{Y discrete state} is observable and controllable and hence, minimal
with McMillian degree $N$.
\end{lemma}
\begin{proof}
First, note that $(A,B,C)$ is minimal due to \autoref{AssumptionB} and \Cref{LemmaMinimality}.
The rest is a consequence of the Kalman-Bertram criterion (B6) and Schlemm and Stelzer \cite[Proposition 3.10]{SchlemmStelzer2012}.
\end{proof}
Plugging $t=nh$ in the definition of $Y(t)$ results in
\beam \label{Y discrete}
    Y_n^{(h)}=C_1X_1(0) +C_1B_1L_n^{(h)}+Y_{st,n}^{(h)}
\eeam
where $(L_n^{(h)})_{n\in\Z}:=(L(nh))_{n\in\Z}$ and
\beam \label{Y_st}
    Y_{st,n}^{(h)}:=Y_{st}(nh)=\int_{-\infty}^{nh}C_2\exp(A_2(t-u))B_2\, \dif L(u)=\sum_{j=-\infty}^nC_2\exp(A_2h(n-j))\xi_{st,j}^{(h)}
\eeam
 is a discrete-time stationary MA process.
This shows that we have also a separation of the stationary part and the integrated part for the discrete-time model $Y^{(h)}$ as for the
continuous-time model $Y$.

\begin{lemma} \label{Lemma 5.1}
The process $Y^{(h)}$ is a cointegrated solution of the discrete-time state space model \linebreak $(\exp(A),\Sigma_\xi^{1/2}C,\xi^*)$
with cointegration space $C_1^\perp$ and
cointegration rank $r=d-c$.
\end{lemma}
\begin{proof}
Is a direct conclusion of  \eqref{Y discrete}, \Cref{Lemma:Appendix} and \Cref{defCointegratedMCARMA}.
\end{proof}

\subsection{Linear innovations of cointegrated state space processes sampled equidistantly}

 The state space representation of $Y^{(h)}$ given in \eqref{Y discrete state} and the representation in \eqref{Y discrete}  have its limits
 for statistical inference since the noise $\xi^{(h)}$, respectively $\xi^*$ is not observable.
For this reason we derive an alternative representation of $Y^{(h)}$, the so called innovation form,
with the help of the linear innovations.

\begin{definition}
\label{defLinearInnovations}
The linear innovations $\varepsilon^{(h)}=(\varepsilon_n^{(h)})_{n\in\N}$ of $Y^{(h)}$ are defined by
\begin{eqnarray*}
\varepsilon_n^{(h)}=Y_n^{(h)}-P_{n-1}Y_n^{(h)},
\end{eqnarray*}
where $P_n$ is the orthogonal projection onto $\overline{\text{span}}\{Y_i^{(h)}:-\infty<i\leq n\}$ and the closure is taken in the Hilbert space of square-integrable random variables with inner product $(X,Y)\mapsto \E \langle X,Y\rangle$.
\end{definition}
To receive the linear innovations of $Y^{(h)}$  we apply the  Kalman filter.
The applicability of the Kalman filter for the sampled cointegrated state-space model $Y^{(h)}$ can be easily checked by adapting the results in  Chui and Chen \cite[Chapter 6]{ChuiChen2009} to the case of unit roots. The complete calculations can be found in Scholz \cite[Section 4.6]{Scholz}.
Important is here as well the observability and controllability of the state space model \eqref{Y discrete state}
 given in \Cref{lemma 5.1a} and assumption (B4).
 In the next proposition, we sum up the basic results concerning the Kalman filter.
\begin{proposition}
\label{propKalmanFilterPropCoint}
The discrete-time algebraic Riccati equation
\begin{align*}
\Omega^{(h)}&:=\mathrm{e}^{Ah}\Omega^{(h)} \mathrm{e}^{A^\mathsf{T}h}-\mathrm{e}^{Ah}\Omega^{(h)} C^\mathsf{T}(C\Omega^{(h)} C^\mathsf{T})^{-1}C\Omega^{(h)}  \mathrm{e}^{A^\mathsf{T}h}+{\Sigma}^{(h)}_\xi\in M_{N,N}(\R) \label{eqRiccatiAlgeqCointmod}
\intertext{has a positive definite solution $\Omega^{(h)}$ and the steady state Kalman gain matrix $K^{(h)}$ is given by}
K^{(h)}&:=\mathrm{e}^{Ah}\Omega^{(h)} C^\mathsf{T}(C\Omega^{(h)} C^\mathsf{T})^{-1} \in M_{N,d}(\R).
\end{align*}
The linear innovations $\varepsilon^{(h)}$ of $Y^{(h)}$ are the unique stationary solution of the linear state-space model
\begin{equation}
\begin{split}
\widehat{X}_n^{(h)}&=\big(\mathrm{e}^{Ah}-K^{(h)}C\big)\widehat{X}_{n-1}^{(h)}+K^{(h)}Y_{n-1}^{(h)},\\
\varepsilon_n^{(h)}&=Y_n^{(h)}-C\widehat{X}_n^{(h)}, \quad n \in\N. \label{eqLinearInnovations1}
\end{split}
\end{equation}
Thus, the innovation form of $Y^{(h)}$ is
\begin{equation}
\begin{split}
\widehat{X}_n^{(h)}&=\e^{Ah}\widehat{X}_{n-1}^{(h)}+K^{(h)}\varepsilon_{n-1}^{(h)},  \\
Y_n^{(h)}&=C\widehat{X}_n^{(h)}+\varepsilon_n^{(h)}, \quad n \in\N.\label{eqInnovationsForm1}
\end{split}
\end{equation}
The covariance matrix of the innovations is given by
\begin{eqnarray*}
\label{eqCovInnov}
V^{(h)}=\E\Big[\varepsilon_n^{(h)}\varepsilon_n^{(h)\mathsf{T}}\Big]=C\Omega^{(h)} C^\mathsf{T}\in\Sympp_d.
\end{eqnarray*}
\end{proposition}
The discrete-time state space model  \eqref{eqInnovationsForm1} is as well minimal with McMillian degree $N$ due to \Cref{lemma 5.1a}.

\subsection{Probabilistic properties of the linear innovations of cointegrated state space processes sampled equidistantly} \label{Section 3.1}

The aim of this subsection is to derive probabilistic properties of the linear innovations  $\varepsilon^{(h)}$.
Since $Y^{(h)}$ is non-stationary it is not obvious which initial condition $\widehat X_1^{(h)}$ results in
the stationarity of $\varepsilon^{(h)}$. We will present an initial condition and a different representation of $\varepsilon^{(h)}$ which as well
 shows nicely that $\varepsilon^{(h)}$ is indeed stationary. From this we directly receive an error correction form.

Therefore, primary  we investigate the properties of the Kalman filter
\begin{eqnarray} \label{eqTransferFunctionKz}
g(z)&:=&I_d-C\big[I_N-\big(\exp(Ah)-K^{(h)}C\big)z\big]^{-1}K^{(h)}z
=I_d-C\sum_{j=1}^\infty\big(\exp(Ah)-K^{(h)}C\big)^{j-1}K^{(h)}z^j \nonumber\\
&=&I_d-\sum_{j=1}^{\infty}G_{j-1}z^{j},\quad z\in\C,
\end{eqnarray}
 with $G_j=C\big(\exp(Ah)-K^{(h)}C\big)^{j}K^{(h)}$. Due to the results known about the Kalman filter  $|\lambda_{max}(\exp(Ah)-K^{(h)}C)|<1$ (for the cointegrated setting see Scholz~\cite[Lemma 4.6.7]{Scholz} and for the stationary case see Chui and Chen, \cite[Lemma 6.8]{ChuiChen2009}), and hence $I_N-\big(\mathrm{e}^{Ah}-K^{(h)}C\big)$ is indeed invertible.
 Rewriting \eqref{eqLinearInnovations1}
 we receive
 \beao
   \varepsilon^{(h)}_n&=&Y_n^{(h)}-C(\e^{A h}-K^{(h)} C)^{n-1} \widehat X_1^{(h)}- \sum_{j=1}^{n-1} C(\e^{A h}-K^{(h)} C)^{j-1}K^{(h)}
     Y_{n-j}^{(h)}\\
    &=& \left[Y_n^{(h)}- \sum_{j=1}^{n-1} G_{j-1} Y_{n-j}^{(h)}\right]-G_{n-1} \widehat X_1^{(h)}.
 \eeao
 We will show that $\varepsilon^{(h)}_n=g(B)Y_{n}^{(h)}$ where $B$ denotes the Backshift operator. In the stationary case this representation
 is straightforward but it is not as obvious for the  non-stationary sequence $Y^{(h)}$. Therefore, we first show that
\begin{eqnarray}
\label{eqTransferFunctionK1}
g(1)=I_d-C\big[I_N-\big(\exp(Ah)-K^{(h)}C\big)\big]^{-1}K^{(h)}\in M_{d}(\R)
\end{eqnarray}
contains the information about the cointegration as cointegration space and rank.
Initially, we adapt  Ribarits and Hanzon, \cite[Lemma 3.1]{RibaritsHanzon2011} in order to characterize the rank of $g(1)$.

\begin{lemma} \label{Lemma 5.5} $\mbox{}$
\begin{itemize}
\item[(a)] $g(1)C_1=0_{d\times c}$.
\item[(b)] $\rank g(1)=d-c.$
\end{itemize}
\end{lemma}
\begin{proof}
First, we write $K^{(h)}=( K_1^{(h)}\,\, K_2^{(h)})$ with $K_1^{(h)}\in M_{N,c}(\R)$ and $K_2^{(h)}\in M_{N,d-c}(\R)$. The matrix $K^{(h)}$ has as product of non-singular matrices and
a matrix with full rank as well full rank $d$. Thus, $K_1^{(h)}$ has full rank $c$ and $K_2^{(h)}$ has full rank $d-c$.
 Applying the decoupling into subsystems to \eqref{eqTransferFunctionK1} we obtain for $g(1)$ the representation
\begin{eqnarray*}
\label{eqDecoupledk1}
g(1)=I_d-\begin{pmatrix}C_1 & C_2\end{pmatrix} \begin{pmatrix}K_1^{(h)}C_1 & K_1^{(h)}C_2\\ K_2^{(h)}C_1 & K_2^{(h)}C_2+I_{N-c}-\exp(A_2h)\end{pmatrix}^{-1} \begin{pmatrix}K_1^{(h)}\\  K_2^{(h)}\end{pmatrix}.
\end{eqnarray*}
 As  mentioned in the beginning of \Cref{Section 3.1} the inverse
in the definition of $g(1)$ is well-defined.
Moreover, since $K_1^{(h)}$ and $C_1$ have full rank $c$, the $c\times c$ matrix $K_1^{(h)}C_1$ is regular and has also rank $c$. Thus, due to Bernstein \cite[Proposition 2.8.3]{Bernstein2009}
we can apply the Matrix Inversion Lemma \cite[Proposition 2.8.7]{Bernstein2009} and obtain
\begin{eqnarray} \label{eq_2}
g(1)=I_d-\begin{pmatrix}C_1 & C_2\end{pmatrix}\cdot \begin{pmatrix}M_{11} & M_{12}\\M_{21} & M_{22} \end{pmatrix}\cdot\begin{pmatrix}K_1^{(h)}\\  K_2^{(h)}\end{pmatrix},
\end{eqnarray}
where
\begin{align*}
Q&:=I_{{N-c}}-\exp({A_2h}) +K_2^{(h)}C_2-K_2^{(h)}C_1(K_1^{(h)}C_1)^{-1}K_1^{(h)}C_2,\\
M_{11}&:=(K_1^{(h)}C_1)^{-1}+(K_1^{(h)}C_1)^{-1}K_1^{(h)}C_2Q^{-1}K_2^{(h)}C_1 (K_1^{(h)}C_1)^{-1},
\\
M_{12}&:=-(K_1^{(h)}C_1)^{-1}K_1^{(h)}C_2Q^{-1},
\\
M_{21}&:=-Q^{-1}K_2^{(h)}C_1(K_1^{(h)}C_1)^{-1},
\\
M_{22}&:=Q^{-1}.
\end{align*}
As already stated above $K_1^{(h)}C_1$ is a regular  $c\times c$ matrix where $K_1^{(h)}$ and $C_1$ have full rank $c$ and consequently, $C_1(K_1^{(h)}C_1)^{-1}K_1^{(h)}$ has rank $c$.
Define $P:=I_d-C_1(K_1^{(h)}C_1)^{-1}K_1^{(h)}\in M_{d,d}(\R)$ which is obviously idempotent since $P^2=P$ holds.
Moreover, $I-P=C_1(K_1^{(h)}C_1)^{-1}K_1^{(h)}\in M_{d,d}(\R)$ has  rank $c$
 so that due to the rank equation for an idempotent matrix,  $\rank P=d-c$ (see  Bernstein \cite[Fact 3.12.9]{Bernstein2009}).
 Then we can rewrite the matrix $g(1)$ once more using  representation \eqref{eq_2} and obtain
\begin{align*}
g(1)=&I_d-\begin{pmatrix}C_1 & C_2\end{pmatrix}\cdot M\cdot\begin{pmatrix}K_1^{(h)} \\ K_2^{(h)}\end{pmatrix}
\\
=&I_d-C_1 M_{11}K_1^{(h)} - C_2 M_{21}K_1^{(h)} -C_1 M_{12}K_2^{(h)} - C_2 M_{22}K_2^{(h)}
\\
=&(I_d-C_1(K_1^{(h)}C_1)^{-1}K_1^{(h)})-C_2Q^{-1}K_2^{(h)}
+C_1(K_1^{(h)}C_1)^{-1}K_1^{(h)}C_2Q^{-1}K_2^{(h)}
\\
&+C_2Q^{-1}K_2^{(h)}C_1(K_1^{(h)}C_1)^{-1}K_1^{(h)}
-C_1(K_1^{(h)}C_1)^{-1}K_1^{(h)}C_2Q^{-1}K_2^{(h)}C_1(K_1^{(h)}C_1)^{-1}K_1^{(h)}
\\
=&P-(I_d-C_1(K_1^{(h)}C_1)^{-1}K_1^{(h)})C_2Q^{-1}K_2^{(h)}(I_d-C_1(K_1^{(h)}C_1)^{-1}K_1^{(h)})
\\
=&P-PC_2Q^{-1}K_2^{(h)}P.
\end{align*}
From the representation of $P$ we receive $PC_1=0_{d\times c}$ and hence, $g(1)C_1=0_{d\times c}$ which is (a).

A further conclusion of the Matrix Inversion Lemma \cite[Corollary 2.8.8]{Bernstein2009}  gives
\begin{align*}
Q^{-1}=&[I_{{N-c}}-\exp({A_2h})+K_2^{(h)}P^2C_2]^{-1}
\\
=&\big(I_{{N-c}}-\exp({A_2h})\big)^{-1}\\
&-\big(I_{{N-c}}-\exp({A_2h})\big)^{-1}K_2^{(h)}P [I_d+PC_2\big(I_{{N-c}}-\exp({A_2h})\big)^{-1}K_2^{(h)}P]^{-1}PC_2\big(I_{{N-c}}-\exp({A_2h})\big)^{-1}.
\end{align*}
For the sake of brevity, we write $R:=C_2\big(I_{{N-c}}-\exp({A_2h})\big)^{-1}K_2^{(h)}$. Substituting the previous result into the formula for $g(1)$ leads to
\begin{align*}
g(1)&=P-PRP+PRP(I_d+PRP)^{-1}PRP
=P-PRP+(PRP)^2(I_d+PRP)^{-1}
\intertext{where we used the fact that $(I_d+AB)^{-1}A=A(I_d+BA)^{-1}$ for matrices $A(:=PRP)$ and $B(:=I_d)$  (see  \cite[Fact 2.16.16]{Bernstein2009}). Then $P^2=P$ results in}
g(1)&=[(P-PRP)(P+PRP)+(PRP)^2](I_d+PRP)^{-1}
=P(I_d+PRP)^{-1}.
\end{align*}
 We can conclude that
$\rank g(1)=\rank P(I_d+PRP)^{-1}=\rank P=d-c$
which is (b).
\end{proof}

\begin{lemma} \label{Lemma 5.6}
Define ${K}_j=\sum_{k=j}^{\infty}G_k=C\sum_{k=j}^\infty\big(\exp(Ah)-K^{(h)}C\big)^{k}K^{(h)}$ for $j\in\N_0$ and
$ k(z)=\sum_{j=1}^\infty  K_jz^j$, $z\in\C$.
Then
\beao
   g(z)= g(1)z+I_d(1-z)+{k}(z)(1-z), \quad z\in\C.
\eeao
\end{lemma}
\begin{proof}
First of all, note that $g(1)=I_d-K_0$ and $ K_j- K_{j+1}=G_j$. Then
\beao
    k(z)(1-z)=\sum_{j=1}^\infty K_jz^j-\sum_{j=1}^\infty K_jz^{j+1}
        =\sum_{j=1}^\infty( K_j- K_{j-1})z^j+ K_0z
        =-\sum_{j=1}^\infty G_{j-1}z^j+ K_0z.
\eeao
Hence,
\beao
    g(1)z+I_d(1-z)+{k}(z)(1-z)=I_dz- K_0z+I_d-I_dz-\sum_{j=1}^\infty G_{j-1}z^j+ K_0z
        =I_d-\sum_{j=1}^\infty G_{j-1}z^j=g(z).
\eeao
\end{proof}

Now we have derived every important property of the Kalman filter $g(z)$ which we require to present the initial condition $\widehat X_1^{(h)}$
giving us the stationary linear innovation $\varepsilon^{(h)}$ and their representation. However, we need some further notation.
It is important to know that the stationary process $Y_{st}(t)$ can be defined for $t\in\R$ as $Y_{st}(t)=\int_{-\infty}^tf_{st}(t-s)\,dL(s)$ with
$f_{st}(u)=C_{2}\exp(A_{2}u)B_{2}\1_{[0,\infty)}(u)$, and $Y_{st,n}^{(h)}$ can be defined for $n\in\Z$ as $Y_{st,n}^{(h)}=Y_{st}(nh)$.
Then we have an adequate definition of $\Delta Y^{(h)}_n:=Y^{(h)}_n-Y^{(h)}_{n-1}$ for negative values $n$ as well as
 $\Delta Y^{(h)}_n=\int_{-\infty}^{nh}f_{\Delta}(nh-s)\,d L(s)$, $n\in\Z$,
with $f_{\Delta}(u)=f_{st}(u)-f_{st}(u-h)+C_1B_1\1_{[0,h)}(u)$.
Finally, we are able to extend $Y^{(h)}_n$ on the negative integers as
\beao
    Y^{(h)}_{-n}:=Y(-nh):=C_{1}X_1(0)+Y_{st,0}^{(h)}-\sum_{j=0}^{n-1}\Delta Y^{(h)}_{-j}
        =C_{1}X_1(0)+L^{(h)}_{-n}+Y_{st,-n}^{(h)} \quad \text{ for } n\in\N_0.
\eeao
\begin{theorem} \label{theorem 5.7}
The initial condition in the Kalman filter is $$\widehat X_1^{(h)}=\sum_{j=0}^{\infty} (\e^{A h}-K^{(h)} C)^{j}K^{(h)}
     Y^{(h)}_{-j}$$
Moreover, there exists an $\alpha, C_1^{\perp}\in M_{d, d-c}(\R)$ with rank $d-c$ so that
\begin{eqnarray*}
\varepsilon_n^{(h)}=g(B)Y_n^{(h)}=-\alpha C_1^{\perp\,\mathsf{T}}Y_{n-1}^{(h)}+\Delta Y_n^{(h)}+{k}(B)\Delta Y_n^{(h)}
\end{eqnarray*}
where $B$ denotes the backshift operator.
\end{theorem}
\begin{proof}
Since $\rank g(1)=d-c$ (see \Cref{Lemma 5.5}(b)) there exists $\alpha,\beta\in M_{d, d-c}(\R)$ with full row rank $d-c$ such that $g(1)=-\alpha\beta^\mathsf{T}$.
Furthermore, due to \Cref{Lemma 5.5}(a), $0_{d\times c}=g(1)C_1=-\alpha\beta^{\mathsf{T}}C_1$. Since $\alpha$ has full rank $d-c$,
$\beta^{\mathsf{T}}C_1=0_{d-c\times c}$ and thus, we denote $\beta$ as $C_1^{\perp}$.
Then \Cref{Lemma 5.6} leads to
\begin{eqnarray*}
Y_n^{(h)}-\sum_{j=1}^\infty G_{j-1}Y_{n-j}^{(h)}&=&g(B)Y_n^{(h)}=\left[g(1)B+I_d(1-B)+{k}(B)(1-B)\right]Y_n^{(h)}\\
                   &=&-\alpha C_1^\mathsf{\perp T}Y_{n-1}^{(h)}+\Delta Y_n^{(h)}+{k}(B)\Delta Y_n^{(h)}\\
                   &=&-\alpha C_1^\mathsf{\perp T}Y_{st,n-1}^{(h)}+\Delta Y_n^{(h)}+{k}(B)\Delta Y_n^{(h)}.
\end{eqnarray*}
From this we see that $(g(B)Y_n^{(h)})_{n\in\N}$ is stationary and with $\widehat X_1^{(h)}$ as given above we receive
$\varepsilon_n^{(h)}=g(B)Y_n^{(h)}$.
\end{proof}
We are now able to state some useful properties of the linear innovations.

\begin{proposition}
\label{PropLinearInnovationsProperties}
The linear innovations $\varepsilon^{(h)}$  are a stationary, ergodic and uncorrelated sequence with finite second moments.
\end{proposition}
Before we prove this proposition we require an auxiliary result.

\begin{lemma} \label{Lemma 5.9}
Let $(\xi_n^{(h)})_{n\in\N_0}$ be given as in \eqref{eqNoiseTermR}.
Then there exists a matrix sequence $(\phi_j)_{j\in\Z}$ with $\phi_j\in M_{d,N}(\R)$ and $\|\phi_j\|\leq C\rho^j$ for some $0<\rho<1$
such that for $\Phi(z)=\sum_{j=0}^\infty\phi_jz^{j}$  we have
\beao
    \varepsilon_n^{(h)}=\Phi(B)\xi_n^{(h)}.
\eeao
\end{lemma}
\begin{proof}
Let us define
\beao
    Y_{1,n}^{(h)}:=C_1X_1(0)+C_1B_1L_n^{(h)}=C_1X_1(0)+C_1\sum_{j=1}^n\xi_{1,j}^{(h)}, \quad n\in\Z,
\eeao
such that $Y_n^{(h)}=Y_{1,n}^{(h)}+Y_{st,n}^{(h)}$.
Define $\psi_j:=C_2\exp(A_2hj)$ for $j\in\N_0$ and $\Psi(z)=\sum_{j=0}^\infty\psi_jz^j$, $z\in\C$. Then we obtain
$Y_{st,n}^{(h)}=\Psi(B)\xi_{st,n}^{(h)}$ by \eqref{Y_st}.
Applying \Cref{theorem 5.7} gives
\beao
    \varepsilon_n^{(h)}&=&g(B)Y_n^{(h)}=g(B)Y_{1,n}^{(h)}+g(B)Y_{st,n}^{(h)}\\
        &=&\Delta Y_{1,n}^{(h)}+{k}(B)\Delta Y_{1,n}^{(h)}+ g(B)Y_{st,n}^{(h)}\\
        &=&C_1\xi_{1,n}^{(h)}+{k}(B)(C_1\xi_{1,n}^{(h)})+ g(B)(\Psi(B)\xi_{st,n}^{(h)}).
\eeao
Thus, we define $\Phi_1(z)=C_1+k(z)$, $\Phi_2(z)=g(\Psi(z))$ and $\Phi(z)=(\Phi_1(z),\,\,\Phi_2(z))$ for $z\in\C$ then
$ \varepsilon_n^{(h)}=\Phi(B)\xi_n^{(h)}$.  It remains to show that the coefficients are exponentially bounded.
Therefore, we use that
$g(z)=I_d-\sum_{j=1}^{\infty}G_{j-1}z^{j}$
 with $G_j=C\big(\exp(Ah)-K^{(h)}C\big)^{j}K^{(h)}$. Since $|\lambda_{max}(\exp(Ah)-K^{(h)}C)|<1$ (see \cite[Lemma 4.6.7]{Scholz}) there exists a
 $0<\rho<1$ so that $\|G_j\|\leq C_1\rho^j$ for $j\in\N_0$. Furthermore, for $K_j=\sum_{k=j}^\infty G_k$ we receive  $\|K_j\|\leq C_2\rho^j$.
 Since $\sigma(A_2)\subseteq \{(-\infty,0)+i\R \}$ we get as well that there exists a $0<\widetilde\rho <1$ so that $\|\psi_j\|\leq C_2\widetilde\rho^j$ for $j\in\N_0$.
 Combining everything gives the statement.
\end{proof}

\begin{proof}[Proof of \Cref{PropLinearInnovationsProperties}]
By construction $\varepsilon^{(h)}$ are stationary, uncorrelated and have finite second moments. Due to the MA representation  given in \Cref{Lemma 5.9} with the iid sequence $(\xi_n^{(h)})_{n\in\N_0}$
and Krengel~\cite[Theorem 4.3 in Chapter 1]{Krengel1985}  the process $\varepsilon^{(h)}$ is as well ergodic.
\end{proof}
Note that if the L\'evy process is a Brownian motion then $Y^{(h)}$ and hence, $\epsilon^{(h)}$  are Gaussian. This even implies that the linear innovations are a sequence of i.i.d. random variables.

Finally, we have also derived the\\ \vspace*{-0.3cm} \\
\textbf{Error Correction Form}:
\begin{eqnarray*} \label{defTFECM}
\Delta Y_n^{(h)}=\alpha C_1^{\perp\,\mathsf{T}}Y_{n-1}^{(h)}-{k}(B)\Delta Y_n^{(h)}+\varepsilon_n^{(h)},\quad n\in\N.
\end{eqnarray*}
This error correction form is similar to the well-known error correction form of Johansen~\cite{Johansen1991} for VAR models
(see as well Yap and Reinsel~\cite{Reinsel1997}).
However, there are only two essential differences making asymptotic theory based on the error correction form
more involved. On the one hand, $\varepsilon^{(h)}$
is  only a weak white noise and not an iid sequence; in general it is not even mixing (cf.
Schlemm and Stelzer \cite{SchlemmStelzer2012} for stationary processes). On the other hand, the filter $k(z)$ is an infinite linear filter
in contrast to the usual finite linear filter as, e.g., in~ \cite{LuetkepohlClaessen,Saikkonen92,YapReinsel95a} for VARMA models and  in \cite{Aoki90,BauerWagner2002}
for discrete-time state space models making calculations more involved. Thus, classical mixing limit results can not be used. The advantage of the error correction form is that it consists only
of stationary processes namely $(\Delta Y_n^{(h)})_{n\in\N}$, $(C_1^{\perp\,\mathsf{T}}Y_{n-1}^{(h)})_{n\in\N}$ and $(\varepsilon_n^{(h)})_{n\in\N}$
but still contains the cointegration information in $C_1^{\perp}$.
This form lays the foundation for statistical inference of cointegrated
state space processes similarly as in the discrete-time case; more details can be found in Fasen and Scholz~\cite{Fasen:Scholz:2016b}.


\section{Conclusion} \label{Outlook}
The paper investigates the probabilistic properties of cointegrated state space processes and consists of three main parts.
The first part characterizes cointegrated solutions of linear state space models and shows that they can be written as the sum of a Lévy process
and a stationary state space process. This characterization provides the basis for any further studies
 of cointegrated state space processes. Furthermore, a canonical form for these processes is given. A main result of the second part is that
the class of cointegrated state space processes and MCARMA processes are equivalent and thereby a different characterization of cointegration is given
using the MCARMA structure. Finally, the last part of this paper derives a linear innovation and an error correction form
for cointegrated state space processes sampled equidistantly using the representation of cointegrated state space processes of the first of the
part paper.

The results of this paper lay the groundwork for statistical inference for cointegrated state space processes sampled equidistantly. On the one
hand, the canonical form is useful to construct a family of identifiable cointegrated state space processes which is essential for estimation.
On the other hand, based on the structure of the linear innovations it is possible to define a quasi-maximum-likelihood function and do maximum-likelihood
estimation. The asymptotic properties of these maximum-likelihood estimators are developed in  Fasen and Scholz~\cite{Fasen:Scholz:2016b}.
They show that the quasi-maximum likelihood estimator for the long-run parameter $C_1$ is super-consistent and asymptotically mixed normally distributed
whereas the quasi-maximum likelihood estimator for the other parameters, the short-run parameters are consistent and asymptotically normally distributed.

\begin{appendix}

\section{Appendix}

\begin{lemma} \label{Lemma:Appendix}
Let $C\in M_{d,c}(\R)$ with $\rank C=c$.
\begin{itemize}
    \item[(a)] Then there exists a unique matrix $U\in M_{d,c}(\R)$ with $U^{\mathsf{T}}U=I_c$ and a lower triangular matrix $L\in GL_c(\R)$
    with positive diagonal entries so that $C=UL$.
    \item[(b)] Then there exists a positive lower triangle matrix $U_L\in M_{d,c}(\R)$ with $U_L^{\mathsf{T}}U_L=I_c$ and a matrix $T\in GL_c(\R)$
so that $C=U_LT$.
\end{itemize}
\end{lemma}
\begin{proof} $\mbox{}$\\
(a) \, is a consequence of Gram-Schmidt orthonormalization. Define the matrix $E=[E_{ij}]_{1\leq i,j\leq c}\in M_{c,c}(\R)$
as $E_{ij}=\1_{\{i+j=n+1\}}$. Then $E^2=I_c$ and  $CE\in M_{d,c}(\R)$ has rank $c$ as well.
Due to \cite[Fact 5.15.9]{Bernstein2009} there exists a unique matrix $U'\in M_{d,c}(\R)$ with $(U')^{\mathsf{T}}U'=I_c$ and an upper triangular matrix $R\in M_{c,c}(\R)$
with positive diagonal entries so that $CE=U'R$. Hence, $C=(U'E)(ERE)=UL$ with $U=U'E$ and $L=ERE$. Moreover, $U^{\mathsf{T}}U=E^{\mathsf{T}}(U')^{\mathsf{T}}U'E=E^2=I_c$
and $L$ is positive lower triangular.\\
(b) \,
Clearly, there exists a $T'\in GL_c(\R)$ such that $CT'=:C'$ is a positive lower triangle matrix (cf. QR decomposition).
Due to (a) there exists a unique matrix $U\in M_{d,c}(\R)$ with $U^{\mathsf{T}}U=I_c$ and a lower triangular matrix $L\in GL_c(\R)$
    with positive diagonal entries so that $C'=UL$. Since both $C'$ and $L$ and hence, $L^{-1}$ are positive lower triangular necessarily $U=C'L^{-1}$ is positive lower triangular as well. Defining $T=L[T']^{-1}\in GL_c(\R)$ and $U_L=U$ we receive $C=U_LT$.
\end{proof}

\end{appendix}

\end{document}